\documentclass[11pt]{amsart}
\pdfoutput=1

\usepackage{amssymb}
\usepackage{amsthm}
\usepackage{amsfonts}
\usepackage{amsmath}
\usepackage{pmboxdraw}
\usepackage{verbatim} % for comment
\usepackage{graphicx}
\usepackage{color}
\usepackage[colorlinks=true, citecolor=blue, filecolor=black, linkcolor=black, urlcolor=black]{hyperref}
\usepackage{cite}
\usepackage[normalem]{ulem}
\usepackage{subcaption}
\usepackage{bbm}
\usepackage{bm}
\usepackage{mathtools}
\mathtoolsset{showonlyrefs}
\usepackage{todonotes}
\newcommand{\RN}[1]{%
	\textup{\uppercase\expandafter{\romannumeral#1}}%
}

\def\bp{{\bar\partial}}

\def\bfs{\boldsymbol}

\def\pa{\partial}

\def\wh{\widehat}

\def\Re{ \mathrm{Re}}

\def\KK{\mathcal{K}}

\def\C{\mathbb{C}}

\def\H{\mathbb{H}}
\def\P{\mathbf{P}}
\def\R{\mathbb{R}}

\newcommand{\Pf}{{\textup{Pf}}}
\newcommand{\erfc}{\operatorname{erfc}}
\newcommand{\erf}{\operatorname{erf}}
\newcommand{\bfR}{\mathbf{R}}

\newcommand{\bfkappa}{{\bm \varkappa}}

\newcommand{\Ai}{\operatorname{Ai}}

\newcommand{\re}{\operatorname{Re}}
\newcommand{\im}{\operatorname{Im}}

\theoremstyle{plain}
\newtheorem*{thm*}{Theorem}
\newtheorem{thm}{Theorem}[section]
\newtheorem{lem}[thm]{Lemma}

\newtheorem{prop}[thm]{Proposition}
\newtheorem*{prop*}{Proposition}
\newtheorem*{lem*}{Lemma}

\theoremstyle{definition}
\newtheorem*{eg*}{Example}
\newtheorem*{egs*}{Examples}
\newtheorem*{def*}{Definition}
\newtheorem*{Q*}{Question}

\theoremstyle{remark}
\newtheorem*{rmk*}{Remark}
\newtheorem*{rmks*}{Remarks}
\newcommand{\abs}[1]{\lvert#1\rvert}
\numberwithin{equation}{section}

\begin{document}
\title[Wronskian structures of planar symplectic ensembles]{Wronskian structures of planar symplectic ensembles}

%%%%%%%%%%%%%%%%%%%%%%%%%%%%% author %%%%%%%%%%%%%%%%%%%%%%%%%%%%
\author{Sung-Soo Byun}
\address{Center for Mathematical Challenges, Korea Institute for Advanced Study, 85 Hoegiro, Dongdaemun-gu, Seoul 02455, Republic of Korea}
\email{sungsoobyun@kias.re.kr}

\author{Markus Ebke}
\address{Department of Mathematics, Friedrich-Alexander-Universität Erlangen-Nürnberg, Cauerstrasse 11, 91058 Erlangen, Germany}
\email{markus.ebke@fau.de}

\author{Seong-Mi Seo}
\address{Department of Mathematics, Chungnam National University, 99 Daehak-ro, Yuseong-gu, Daejeon 34134, Republic of Korea.}
\email{smseo@cnu.ac.kr}
%%%%%%%%%%%%%%%%%%%%%%%%%%%%% author %%%%%%%%%%%%%%%%%%%%%%%%%%%%

\thanks{Sung-Soo Byun was partially supported by Samsung Science and Technology Foundation (SSTF-BA1401-51), by the National Research Foundation of Korea (NRF-2019R1A5A1028324) and by a KIAS Individual Grant (SP083201) via the Center for Mathematical Challenges at Korea Institute for Advanced Study.
Markus Ebke was partially supported by Deutsche Forschungsgemeinschaft (IRTG 2235).
Seong-Mi Seo was partially supported by the KIAS Individual Grant (MG063103) at Korea Institute for Advanced Study and
by the National Research Foundation of Korea (2019R1A5A1028324, 2019R1F1A1058006).
}

%\date{\today}

%%%%%%%%%%%%%%%%%%%%%%Abstract%%%%%%%%%%%%%%%%%%%%%%%%%%%%%%%%%%%
\begin{abstract}
We consider the eigenvalues of non-Hermitian random matrices in the symmetry class of the symplectic Ginibre ensemble, which are known to form a Pfaffian point process in the plane.
It was recently discovered that the limiting correlation kernel of the symplectic Ginibre ensemble in the vicinity of the real line can be expressed in a unified form of a Wronskian.
We derive scaling limits for variations of the symplectic Ginibre ensemble and obtain such Wronskian structures for the associated universality classes.
These include almost-Hermitian bulk/edge scaling limits of the elliptic symplectic Ginibre ensemble and edge scaling limits of the symplectic Ginibre ensemble with boundary confinement.
Our proofs follow from the generalised Christoffel-Darboux formula for the former and from the Laplace method for the latter.
Based on such a unified integrable structure of Wronskian form, we also provide an intimate relation between the function in the argument of the Wronskian in the symplectic symmetry class and the kernel in the complex symmetry class which form determinantal point processes in the plane.
\end{abstract}
%%%%%%%%%%%%%%%%%%%%%%Abstract%%%%%%%%%%%%%%%%%%%%%%%%%%%%%%%%%%%

\maketitle
%\tableofcontents
%%%%%%%%%%%%%%%%%%%%%%%introduction%%%%%%%%%%%%%%%%%%%%%%%%%%%%%%%%%

\section{Introduction}
In the study of integrable models in random matrix theory, the underlying structure is one of the most important features that makes asymptotic analysis possible.
%The integrable structure of random matrix theory is perhaps one of the most important features that makes asymptotic analysis possible.
We shall study such a structure for the planar symplectic ensemble $\bfs{\zeta}= \{ \zeta_j \}_{j=1}^N$ whose joint probability distribution follows
\begin{equation}\label{Gibbs}
d\P_N(\boldsymbol{\zeta}) = \frac{1}{Z_N} \prod_{j>k=1}^{N} \abs{\zeta_j-\zeta_k}^2 \abs{\zeta_j-\overline{\zeta}_k}^2 \prod_{j=1}^{N} \abs{\zeta_j-\overline{\zeta}_j}^2 e^{ -N  Q(\zeta_j) } \,  dA(\zeta_j),  \qquad (dA(\zeta):=\tfrac{1}{\pi}d^2\zeta),
\end{equation}
where $Z_N$ is the partition function.
Here $Q: \C \to \R$ is a suitable function satisfying the complex conjugation symmetry $Q(\zeta)=Q(\bar{\zeta})$, called the external potential.
For the special case when $Q(\zeta)=|\zeta|^2$, the measure \eqref{Gibbs} describes the distribution of the eigenvalues of quaternionic Gaussian random matrices, also known as the symplectic Ginibre ensemble \cite{ginibre1965statistical}.

It is well known \cite{MR2934715} that as $N \to \infty,$ the system $\bfs{\zeta}$ tends to occupy a certain compact subset $S \subset \C$ called the droplet.
We shall study the local statistics of the ensemble \eqref{Gibbs} in the vicinity of the real line.
For this purpose, we define the rescaled point process $\boldsymbol{z}=\{z_1,\dots,z_N\}$ as
\begin{equation} \label{rescaling}
z_j:=e^{-i \theta} \gamma_N^{-1} \cdot (\zeta_j-p), \qquad p \in S \cap \R,
\end{equation}
where $\gamma_N$ is an appropriate $N$-dependent factor called the microscopic scale, cf. \eqref{microscale Laplacian} and \eqref{microscale HW}.
Here $\theta\in\R$ is the angle of the outward normal direction at the boundary if $p\in\partial S$ and otherwise $\theta=0$.

By definition, the $k$-point correlation function $R_{N,k}$ is the normalised probability that $k$ of the eigenvalues lie in infinitesimal neighbourhoods of $z_1,\dots, z_k$, i.e.
\begin{equation}
R_{N,k}(z_1,\dots,z_k)
:= \lim_{\varepsilon \downarrow 0} \frac{ \mathbb{P}( \, \exists \text{ at least one particle in } \mathbb{D}(z_j,\varepsilon), j=1,\cdots,k ) }{\varepsilon^{2k}},
\end{equation}
see also \eqref{bfRNk def} for a more standard definition.
Recall that for a $2n \times 2n$ skew-symmetric matrix $A=(a_{j,k})$, its Pfaffian $\Pf(A)$ is given by
$$
\Pf(A)= \frac{1}{2^n n!} \sum_{ \sigma \in S_{2n} } \textup{sgn}(\sigma) \prod_{j=1}^n a_{ \sigma(2j-1),\sigma(2j) },
$$
where $S_{2n}$ is the symmetric group of order $(2n)!$.
Recently, it was shown in \cite{akemann2021scaling} that for the symplectic Ginibre ensemble, when $Q(\zeta)=|\zeta|^2$ and thus $S=\mathbb{D}(0,\sqrt{2})$, the $k$-point correlation function $R_{N,k}(z_1,\dots,z_k)$ converges, as $N \to \infty$, locally uniformly to the limit
\begin{equation} \label{R Pfaffian}
R_k(z_1,\dots,z_k)=\prod_{j=1}^k (\bar{z}_j-z_j) \Pf \Big[ e^{-|z_j|^2-|z_l|^2}
\begin{pmatrix}
\kappa(z_j,z_l) & \kappa(z_j,\bar{z}_l)
\smallskip
\\
\kappa(\bar{z}_j,z_l) & \kappa(\bar{z}_j,\bar{z}_l)
\end{pmatrix}
\Big]_{j,l=1}^k,
\end{equation}
where the pre-kernel $\kappa$ is of the form
\begin{equation} \label{kappa Wronskian}
	\kappa(z,w):=\sqrt{\pi} e^{z^2+w^2} \int_{E} W(f_{w},f_{z})(u) \, du.
\end{equation}
Here, $W(f,g):=fg'-gf'$ is the Wronskian, and
  \begin{equation} \label{fE NH bulk edge}
    	f_z(u):=\tfrac12 \erfc(\sqrt{2}(z-u)), \qquad
    	E:=\begin{cases}
    	    (-\infty,\infty) &\textup{for the bulk case},
    	    \smallskip
    	    \\
    	    (-\infty,0) &\textup{for the edge case}.
    	\end{cases}
    \end{equation}
(See \cite{BL2} for a related result.
We also refer to \cite{MR1986426,dubach2021symmetries} for the spectral radius of the symplectic Ginibre ensemble.)

We emphasise that the function $f_z$ is not necessarily unique.
For instance, the variation $\tilde{f}_z(u) := a f_z(u) + b(z)$ leads to the same pre-kernel in \eqref{kappa Wronskian} if $a = \pm 1$ and $b(z) = f_z(u)\vert_{\partial E}$.
Thus, instead of $f_z$ in \eqref{fE NH bulk edge}, one may take
\begin{equation} \label{f NH bulk edge alternative}
\begin{cases}
    f_z(u) := \tfrac12 \erf(\sqrt{2}(z-u)) & \text{for the bulk case}, \\
    f_z(u) := \tfrac12 \Big( \erfc(\sqrt{2}(z-u))-\erfc(\sqrt{2}z) \Big) & \text{for the edge case}.
\end{cases}
\end{equation}
Note in particular that the $1$-point function $R \equiv R_1$ is given by
\begin{equation} \label{R one pt f}
R(z)=(\bar{z}-z) e^{-2|z|^2} \kappa(z,\bar{z})
=  4\sqrt{\pi}\,\im z \, e^{-4(\im z)^2} \int_{E} \im [f_z'(u) f_{\bar{z}}(u)] \, du.
\end{equation}
For the symplectic Ginibre ensemble, it follows from $S=\mathbb{D}(0,\sqrt{2})$ that the real bulk case corresponds to the regime $p\in (-\sqrt{2},\sqrt{2})$, whereas the real edge case corresponds to the regime $p=\pm \sqrt{2}$.
We also remark that for the real bulk case, one can evaluate the integral in \eqref{kappa Wronskian} in terms of the error function, which corresponds to the representation of the pre-kernel in \cite{MR1928853,akemann2019universal}, see \cite[Remark 2.3.(ii)]{akemann2021scaling} for further details.

Let us now define
\begin{equation}  \label{K complex}
\mathcal{K}(z,w) := 2 \sqrt{\pi} \, e^{z^2+\bar{w}^2} \int_{E} f'_{z}(u)f'_{\bar{w}}(u)\,du.
\end{equation}
With the choice of \eqref{fE NH bulk edge}, it is easy to see that the function $\KK$ in \eqref{K complex} evaluates to
    \begin{equation} \label{K complex Ginibre}
  e^{-|z|^2-|w|^2}  \KK(z,w)=
    \begin{cases}
    2\,e^{-|z|^2-|w|^2+2z\bar{w}} &\textup{for the bulk case},
   \smallskip
   \\
    2\,e^{-|z|^2-|w|^2+2z\bar{w}} \erfc(z+\bar{w}) &\textup{for the edge case}.
    \end{cases}
    \end{equation}
Here, one may notice that \eqref{K complex Ginibre} corresponds to the limiting local correlation kernel of the complex Ginibre ensemble whose joint probability distribution is given by
\begin{align} \label{Gibbs complex}
\begin{split}
d\boldsymbol{\mathcal{P}}_N(\boldsymbol{\lambda}) &= \frac{1}{\mathcal{Z}_N} \prod_{j>k=1}^{N} \abs{\lambda_j-\lambda_k}^2  \prod_{j=1}^{N} e^{ -N  Q(\lambda_j) } \,  dA(\lambda_j)
\end{split}
\end{align}
with $Q(\zeta)=\abs{\zeta}^2$, which forms a determinantal point process, see e.g.~\cite{forrester2010log}.
In other words, a properly defined limiting local correlation function $\wh{R}_k(\wh{z}_1,\dots,\wh{z}_k)$ of the ensemble \eqref{Gibbs complex} is given by
\begin{equation} \label{Rk det}
\wh{R}_k(\wh{z}_1,\dots,\wh{z}_k)=\det \Big[ e^{-|\widehat{z}_j|^2-|\widehat{z}_l|^2} \KK(\widehat{z}_j,\widehat{z}_l) \Big]_{j,l=1}^k.
\end{equation}
In what follows, we call the ensemble \eqref{Gibbs complex} the complex counterpart to the symplectic ensemble \eqref{Gibbs}.

\medskip

The above discussions already lead us to the following natural questions.
\begin{itemize}
    \item \textbf{Question 1.} For a general potential $Q$ in the symplectic ensemble \eqref{Gibbs}, does the limiting pre-kernel $\kappa$ at $p\in\R$ have the Wronskian structure \eqref{kappa Wronskian} for some function $f_z$ and $E \subset \R$?
    \smallskip
    \item \textbf{Question 2.} If the answer to the first question is affirmative, is the function defined in \eqref{K complex} the kernel of the corresponding complex ensemble? cf. \eqref{Rk det}.
\end{itemize}
The main purpose of this study is to examine these questions for the elliptic Ginibre ensemble in the almost-Hermitian regime and the Ginibre ensemble with boundary confinements.

\section{Discussions of main results}

Beyond the standard scaling limits \eqref{fE NH bulk edge} arising from the symplectic Ginibre ensemble, we aim to derive further universality classes whose complex counterparts have been well studied.
In this section, we introduce our models and state the main results.

Until further notice, the microscopic scale in \eqref{rescaling} is given by
\begin{equation}\label{microscale Laplacian}
\gamma_N=\frac{1}{\sqrt{N \delta}}, \qquad   \delta:=\dfrac{\Delta Q(p)}{2}, \qquad (\Delta:=\pa \bp).
\end{equation}
This specific choice of the rescaling comes from the fact that the macroscopic density of $\bfs{\zeta}$ with respect to the area measure is $\frac12 \Delta Q,$ see \cite{MR2934715,ST97}.

\subsection{Symplectic elliptic Ginibre ensemble in the almost-Hermitian regime}

We first study the \emph{symplectic elliptic Ginibre ensemble}, a one-parameter family of random matrices indexed by a non-Hermiticity parameter $\tau \in [0,1)$.
This model describes an interpolation between the Gaussian symplectic ensemble $(\tau=1)$ and the symplectic Ginibre ensemble $(\tau=0)$.
Its eigenvalue statistics correspond to \eqref{Gibbs} with the potential
\begin{equation}\label{Q elliptic}
Q(\zeta):=\tfrac{1}{1-\tau^2}(|\zeta|^2-\tau \re \zeta^2).
\end{equation}
As the terminology ``elliptic'' indicates, the associated droplet $S$ is given by the ellipse
\begin{equation} \label{droplet ellipse}
S:= \{ x+iy \in \C \, | \, ( \tfrac{x}{\sqrt{2}(1+\tau)} )^2+( \tfrac{y}{\sqrt{2}(1-\tau)} )^2 \le 1 \}.
\end{equation}

From the general universality principle of random matrix theory, it can be expected that for any fixed $\tau \in [0,1)$, the local statistics of the ensemble in the large system coincide with the one \eqref{fE NH bulk edge} obtained from the case $\tau=0$.
Based on the skew-orthogonal Hermite polynomials introduced by Kanzieper \cite{MR1928853}, such a statement was recently proved in \cite{akemann2021skew} for $p=0$ and in \cite{byun2021universal} for general $p \in \R$.
(We mention that the local statistic at $p=0$ can be special for certain random matrix models \cite{byun2022almost,byun2022spherical,MR3066113,MR3279619,MR2881072}.)

On the other hand, when $\tau \to 1$ as $N \to \infty$ with an appropriate rate, the limiting correlation functions obtained from the double scaling limit are no longer described in terms of \eqref{fE NH bulk edge}, and it is natural to expect a non-trivial transition between the scaling limits of non-Hermitian and Hermitian random matrices.
In general, such a transition appears in the so-called \emph{almost-Hermitian regime} (or \emph{weak non-Hermiticity}) introduced in the series of works \cite{MR1634312,fyodorov1997almost,MR1431718} by Fyodorov, Khoruzhenko, and Sommers.

\begin{figure}[h!]
		\begin{center}
			\includegraphics[width=0.8\textwidth]{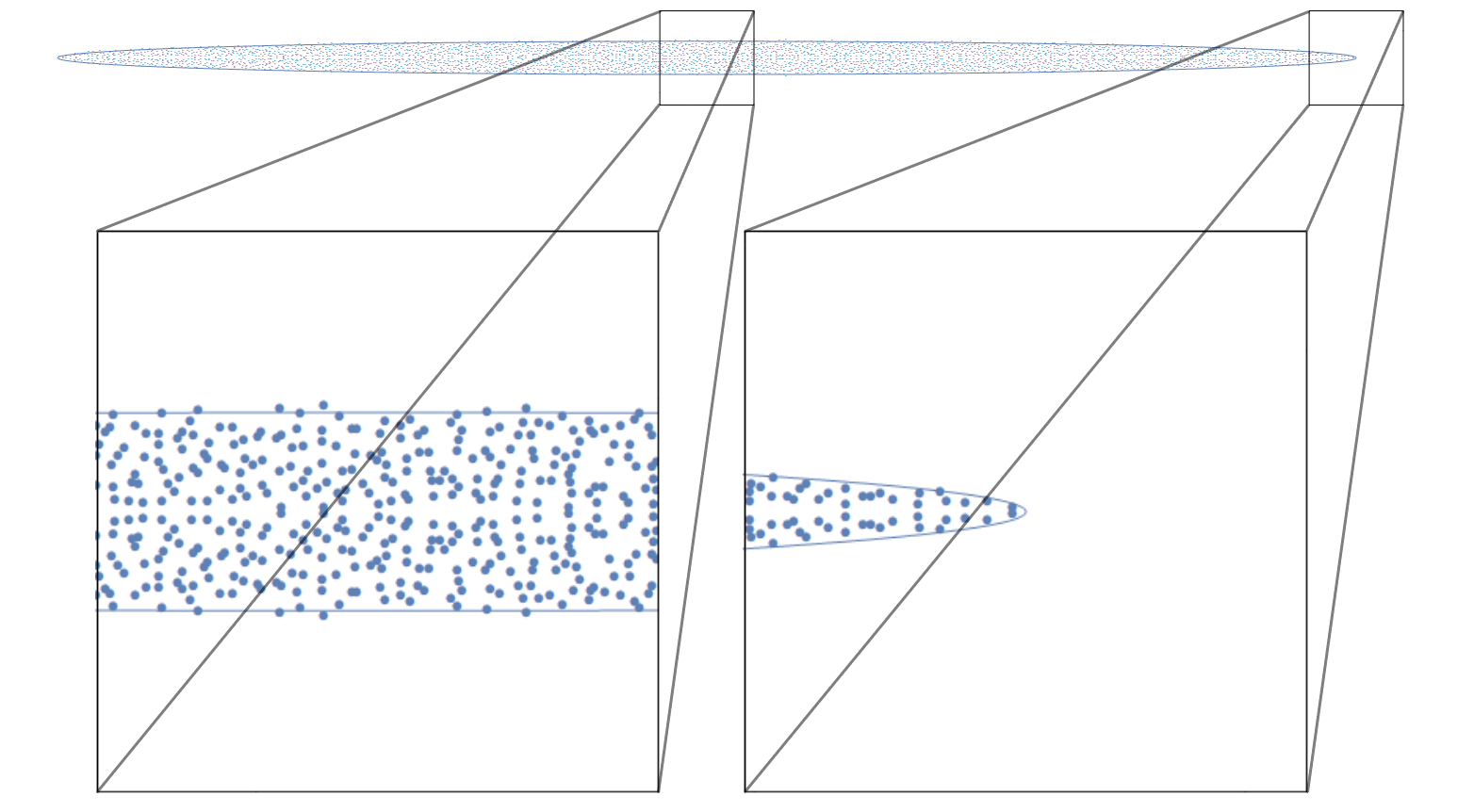}
		\end{center}
	\caption{Eigenvalues of the elliptic Ginibre ensemble in the almost-Hermitian regime} \label{Fig_QEGAH}
\end{figure}

In Theorems~\ref{Thm_AH bulk} and ~\ref{Thm_AH edge} below, we state the scaling limits of the symplectic elliptic Ginibre ensembles in the almost-Hermitian regime, see Figure~\ref{Fig_QEGAH}.
We emphasise that for the special cases $p=0$ and $p=\pm \sqrt{2}(1+\tau)$, (cf.~\eqref{droplet ellipse}) the scaling limits were studied in \cite{MR1928853} and \cite{MR3192169} respectively, based on the idea of Riemann sum approximation of skew-orthogonal Hermite polynomial kernels.
This method is very useful in finding an explicit formula of the limiting pre-kernel, but it is not easy to perform the asymptotic  analysis for general $p \in \R$ or to precisely control the error term.

Instead, we exploit the generalised Christoffel-Darboux formula introduced in \cite{byun2021universal}, which allows us to obtain unified proofs for any points on the real line and to perform precise asymptotic analysis.
(Indeed, this method can also be used to derive the subleading correction terms as well, see \cite{byun2021universal}.)
Furthermore, we describe the scaling limits in terms of the unified Wronskian form \eqref{kappa Wronskian} and show that the associated kernels $\KK$ of the form \eqref{K complex} correspond to those obtained from their complex counterparts.
These provide affirmative answers to the two questions in the previous section for the almost-Hermitian symplectic ensembles.

Our first main result is on the bulk scaling limit of the symplectic elliptic Ginibre ensemble in the almost-Hermitian regime.

\begin{thm} \label{Thm_AH bulk} \textbf{\textup{(Almost-Hermitian bulk scaling limit)}} Let $Q$ be the elliptic potential \eqref{Q elliptic} with
\begin{equation} \label{tau AH bulk}
\tau \equiv \tau_N=1-\frac{c^2}{2N}, \qquad c>0.
\end{equation}
Then for $p \in (-\sqrt{2}(1+\tau),\sqrt{2}(1+\tau))$, the $k$-point correlation function $R_{N,k}$ converges, as $N \to \infty$, locally uniformly to the limit $R_k$ of the form \eqref{R Pfaffian} with
 \begin{equation} \label{fE AH bulk}
    f_z(u):=\frac{1}{2\pi} \int_{I} e^{-t^2/2} \sin(2t(z-u))\,\frac{dt}{t}, \qquad E:=\R.
    \end{equation}
Here
\begin{equation} \label{I tilde c}
I:=(-\tilde{c},\tilde{{c}}),\qquad \tilde{c}:=c \sqrt{1-\tfrac{p^2}{8}}.
\end{equation}
\end{thm}

With the choice \eqref{fE AH bulk}, the pre-kernel $\kappa$ has an alternative representation
    \begin{equation} \label{kappa t.i.}
		\kappa(z,w)=  \frac{1}{\sqrt{\pi}}  e^{ z^2+w^2 }  \int_I e^{-u^2} \sin(2u(z-w) ) \frac{du}{u},
	\end{equation}
see Subsection~\ref{Subsec_AH bulk}.
For $p=0$, this form of the limiting pre-kernel was investigated in \cite{MR1928853}.
We emphasise that by means of Ward's equation, it was shown in \cite[Theorem 2.10]{akemann2021scaling} that   the pre-kernel of the form \eqref{kappa t.i.} is a unique translation invariant scaling limit of general planar symplectic ensembles.

Note in particular that with \eqref{fE AH bulk}, the kernel $\KK$ in \eqref{K complex} is given by
\begin{equation} \label{K AH bulk}
\mathcal{K}(z,w) = \frac{2}{\sqrt{\pi}} \, e^{z^2+\bar{w}^2}\int_I e^{-t^2} \cos(2t(z-\bar{w})) \,dt.
\end{equation}
The kernel \eqref{K AH bulk} corresponds to the one obtained from the complex elliptic Ginibre, see e.g.~\cite{fyodorov1997almost,akemann2016universality}.
As one may notice, the factor $\sqrt{1-p^2/8}$ in \eqref{I tilde c} corresponds to the density of the semi-circle law.
We refer to \cite{AB21} for the geometric interpretation of such a density term, which follows from the limiting shape of the droplet (also called cross-section convergence).

It is also easy to show that the function $f_z$ in \eqref{fE AH bulk} satisfies
\begin{equation}
	\lim_{c\to\infty}	f_{z}(u)= \tfrac12 \erf(\sqrt{2}(z-u)).
\end{equation}
Thus one can recover \eqref{fE NH bulk edge} (cf.~\eqref{f NH bulk edge alternative}) for the bulk case in the non-Hermitian limit when $c\to\infty$.

\begin{figure}[h!]
	\begin{subfigure}{0.32\textwidth}
		\begin{center}
			\includegraphics[width=\textwidth]{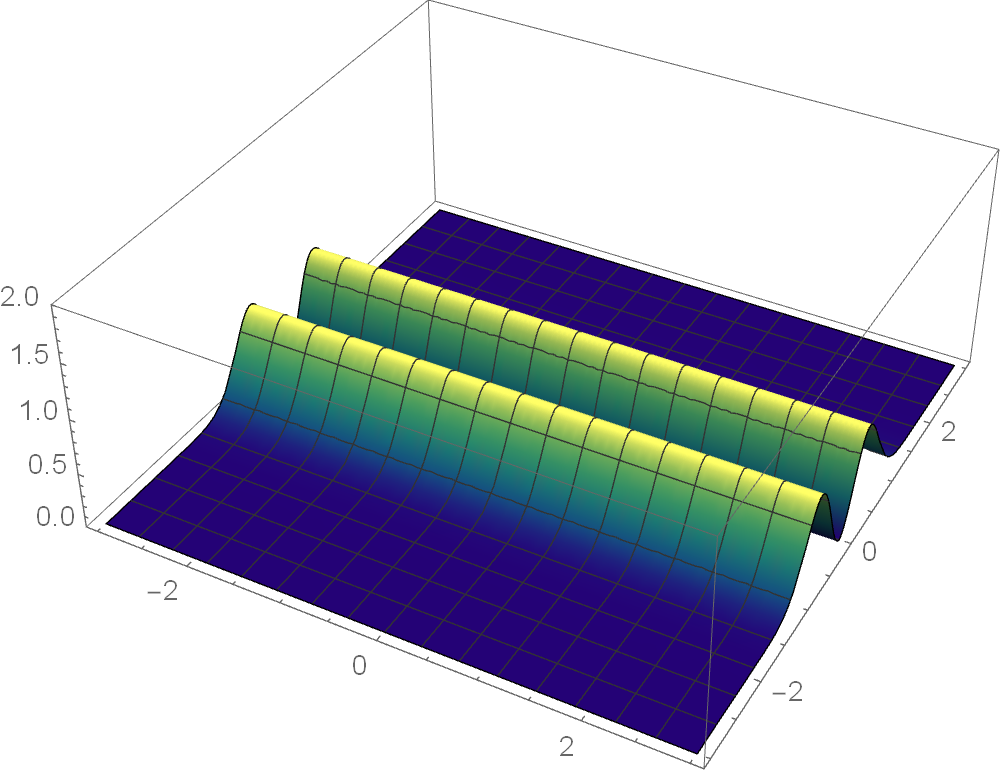}
		\end{center}
		\subcaption{$\tilde{c}=1$}
	\end{subfigure}
	\begin{subfigure}{0.32\textwidth}
		\begin{center}
			\includegraphics[width=\textwidth]{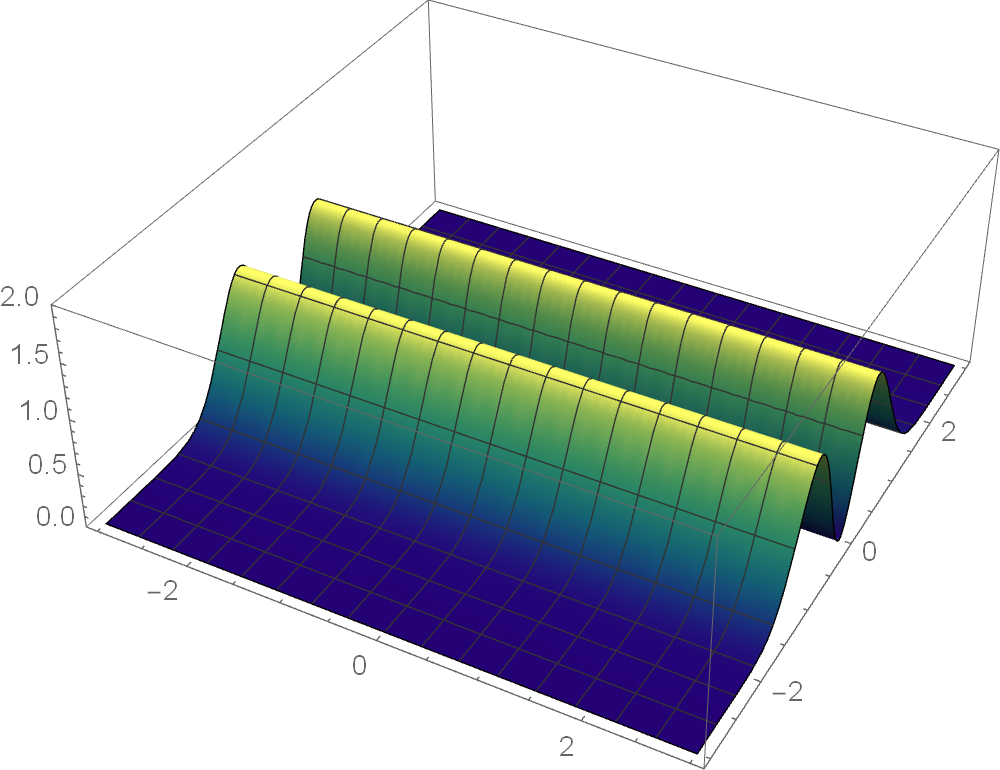}
		\end{center} \subcaption{$\tilde{c}=2$}
	\end{subfigure}
		\begin{subfigure}{0.32\textwidth}
		\begin{center}
			\includegraphics[width=\textwidth]{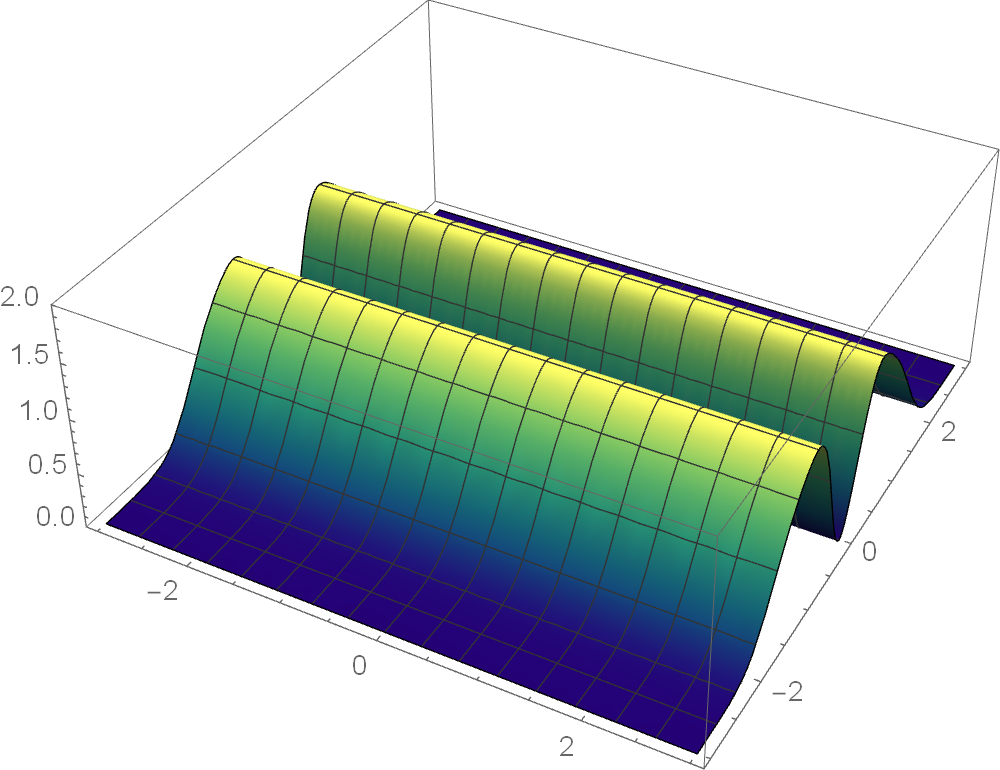}
		\end{center} \subcaption{$\tilde{c}=3$}
	\end{subfigure}
	\caption{ The graphs of the one-point function $R$ in \eqref{R one pt f} with \eqref{fE AH bulk}. } \label{Fig_R AH bulk}
\end{figure}

We now turn to the edge scaling limit.

\begin{thm} \label{Thm_AH edge} \textbf{\textup{(Almost-Hermitian edge scaling limit)}} Let $Q$ be the elliptic potential \eqref{Q elliptic} with
\begin{equation} \label{tau AH edge}
\tau \equiv \tau_N=1-\frac{c^2}{(2N)^{1/3}}, \qquad c>0.
\end{equation}
Then for $p = \pm \sqrt{2}(1+\tau)$, the $k$-point correlation function $R_{N,k}$ converges, as $N \to \infty$, locally uniformly to the limit $R_k$ of the form \eqref{R Pfaffian} with
 \begin{equation}  \label{fE AH edge}
    	f_z(u):=2c\int_{0}^{u} e^{ c^3(z-t)+\frac{c^6}{12} } \Ai\Big(2c(z-t)+\frac{c^4}{4}\Big)\,dt, \qquad E:=(-\infty,0).
    \end{equation}
\end{thm}

Notice that with \eqref{fE AH edge}, the kernel $\KK$ in \eqref{K complex} is given by
\begin{equation} \label{K AH edge}
\mathcal{K}(z,w) = 8 \sqrt{\pi} \,c^2\, e^{z^2+\bar{w}^2+\frac{c^6}{6} }
 \int_{-\infty}^0 e^{ c^3(z+\bar{w}-2u) } \Ai\Big(2c(z-u)+\frac{c^4}{4}\Big)\Ai\Big(2c(\bar{w}-u)+\frac{c^4}{4}\Big) \,du.
\end{equation}
Again, the kernel \eqref{K AH edge} corresponds to the complex counterpart obtained first in \cite{bender2010edge,akemann2010interpolation}.
(We also refer the reader to \cite{AB21,MR3192169} for alternative derivations of \eqref{K AH edge}.)

We now briefly discuss the non-Hermitian  limit when $c \to \infty$.
It follows from the asymptotic behaviour of the Airy function (see e.g.~\cite[Eq.(9.7.5)]{olver2010nist})
\begin{equation} \label{Airy asym}
\Ai(z)=\frac{\exp(-\frac23 z^{3/2})}{2\sqrt{\pi} z^{1/4}} \cdot (1+O(z^{-3/2})) , \qquad (z\to \infty)
\end{equation}
that the function $f_z$ in \eqref{fE AH edge} satisfies
\begin{equation}
	\lim_{c\to\infty}	f_{z}(u)=\tfrac12 \Big( \erfc(\sqrt{2}(z-u))-\erfc(\sqrt{2}z)  \Big).
\end{equation}
Thus one can recover \eqref{fE NH bulk edge} (cf. \eqref{f NH bulk edge alternative}) for the edge case as well.

\begin{figure}[h!]
	\begin{subfigure}{0.32\textwidth}
		\begin{center}
			\includegraphics[width=\textwidth]{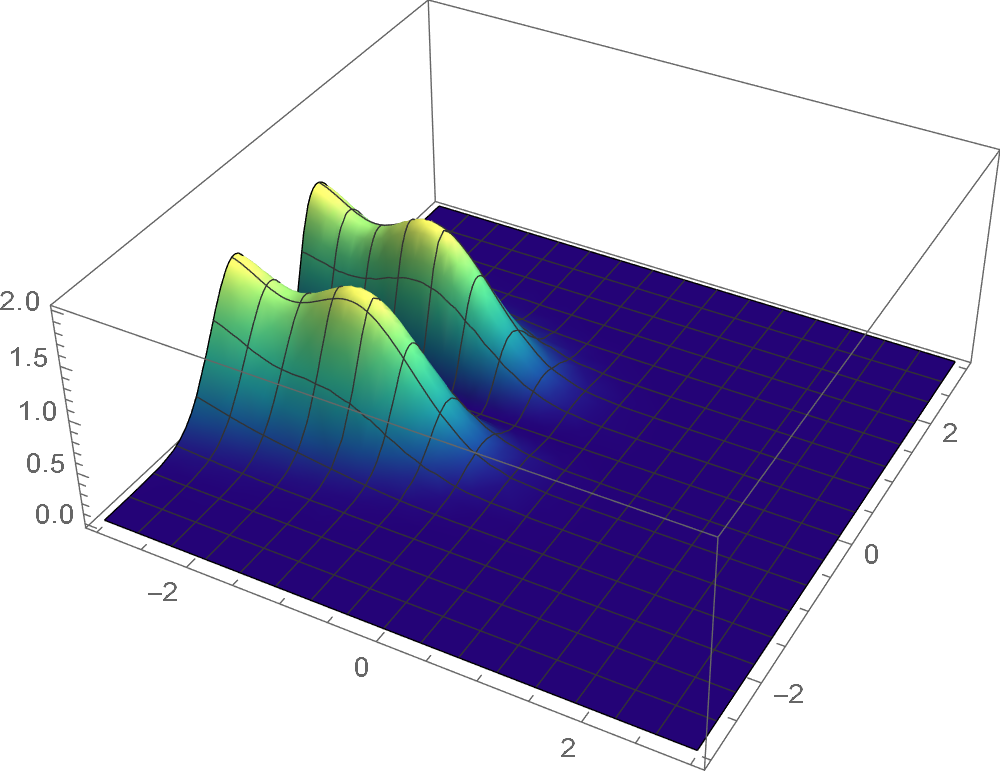}
		\end{center}
		\subcaption{$c=1$}
	\end{subfigure}
	\begin{subfigure}{0.32\textwidth}
		\begin{center}
			\includegraphics[width=\textwidth]{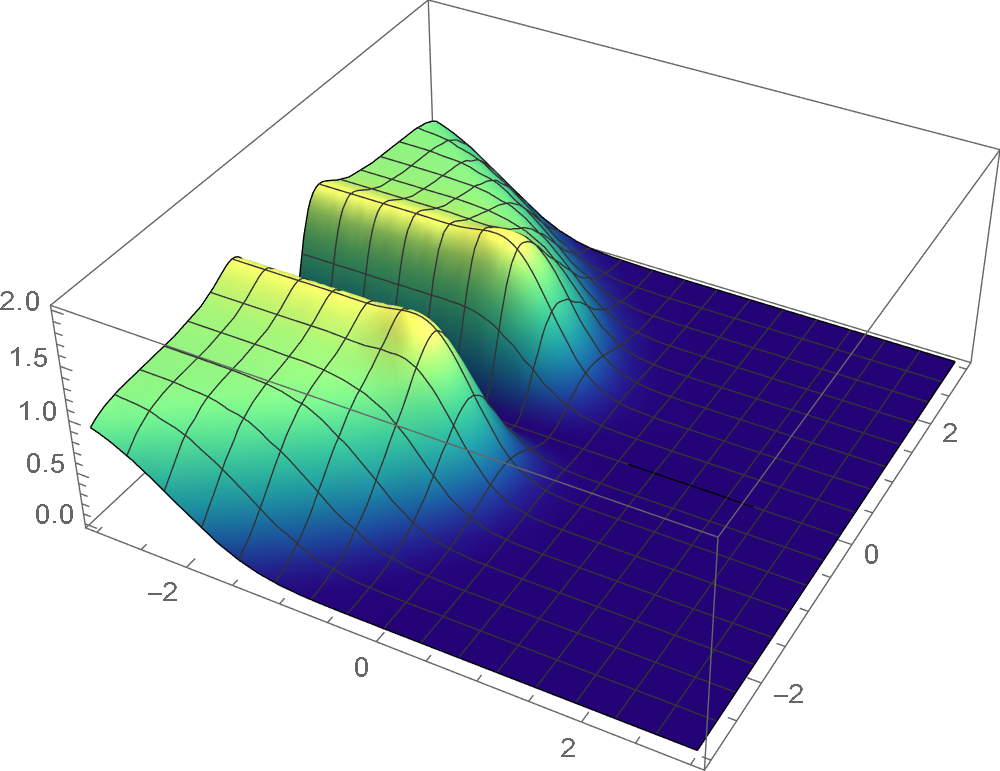}
		\end{center} \subcaption{$c=2$}
	\end{subfigure}
		\begin{subfigure}{0.32\textwidth}
		\begin{center}
			\includegraphics[width=\textwidth]{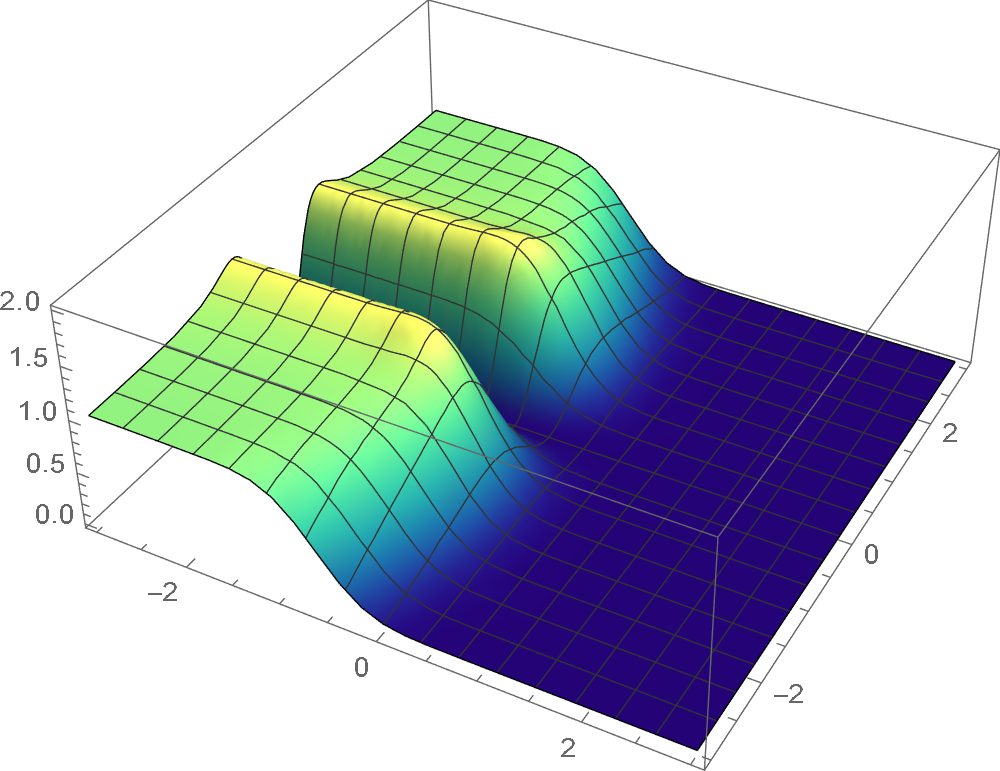}
		\end{center} \subcaption{$c=3$}
	\end{subfigure}
	\caption{ The graphs of the one-point function $R$ in \eqref{R one pt f} with \eqref{fE AH edge}. } \label{Fig_R AH edge}
\end{figure}

\subsection{Symplectic Ginibre ensemble with boundary confinements}

In the previous subsection we have discussed the ensembles \eqref{Gibbs} whose potential does not have any constraints near the boundary. In what follows, we call such a situation as \emph{free boundary} (or \emph{soft edge}) condition.
On the other hand different universality classes naturally arise at the edge, if appropriate boundary constraints are imposed.
Typical examples of such boundary conditions are the so-called \emph{soft/hard edge} and \emph{hard edge} constraints.

In the soft/hard edge setting, we completely confine the particles inside of the droplet $S$ by redefining the potential $Q(\zeta)=+\infty$ outside $S$.
This type of boundary confinement does not change the limiting spectral distribution.
The term soft/hard comes from this situation being called when ``the soft edge meets the hard edge'' \cite{claeys2008universality}.
For the complex ensembles \eqref{Gibbs complex}, such a boundary condition has been investigated in \cite{MR4030288,MR4169375,MR2921180,MR4179777} for example.

In the hard edge setting, we confine the particles further inside of the droplet $S$.
This leads to a modified associated equilibrium measure, in particular giving rise to some non-trivial measure on a certain one-dimensional subset.
From a statistical physics point of view, this confinement has the effect of condensing a non-trivial portion of the particles onto the hard edge.
We refer to  \cite{seo2020edge,hedenmalm2020riemann,MR4071093} and references therein for the studies of complex ensembles \eqref{Gibbs complex} with such type of boundary confinement.
(See also \cite{MR1748745} for a similar situation in the context of truncated unitary ensembles.)

To our knowledge, the edge scaling limits of the symplectic ensembles associated with the above boundary conditions have not been investigated, and we aim to contribute to these problems.
In particular, in Theorems~\ref{Thm_NH hE} and \ref{Thm_NH hW}, we derive the scaling limits of the symplectic Ginibre ensembles with soft/hard edge and hard edge constraints, which provide new universality classes.
See Figure~\ref{Fig_R edge FH} for the graphs of the corresponding one-point functions.
Furthermore, we shall show that the limiting pre-kernels are again of Wronskian form and that the relation to their complex counterparts again holds.

First, we consider the soft/hard edge Ginibre ensemble, which corresponds to the configuration \eqref{Gibbs} with the potential
     \begin{equation} \label{Q Ginibre HE}
	Q(\zeta)=
	\begin{cases}
		|\zeta|^2 &\text{if } |\zeta|\le\sqrt{2},
		\\
		\infty &\text{otherwise}.
	\end{cases}
\end{equation}
By construction, all the eigenvalues are completely confined inside the droplet $S=\mathbb{D}(0,\sqrt{2})$.
As a result, the rescaled point process \eqref{rescaling} at the real edge of the spectrum lies only in the left-half plane $\mathbb{H}_-$.
Since the limiting spectral distribution (the circular law) is the same as the usual symplectic Ginibre ensemble, we rescale the process with the choice of microscopic scale \eqref{microscale Laplacian}.
We then obtain the following.

\begin{thm} \label{Thm_NH hE}\textbf{\textup{(Non-Hermitian soft/hard edge scaling limit)}} Let $Q$ be the soft/hard edge Ginibre potential \eqref{Q Ginibre HE}.
Then for $p = \pm \sqrt{2}$, the $k$-point correlation function $R_{N,k}$ converges, as $N \to \infty$, locally uniformly to the limit $R_k$ of the form \eqref{R Pfaffian} with
  \begin{equation} \label{fE NH Hedge}
    	f_z(u):=\frac{2}{\sqrt{\pi}} \int_{-\infty}^{u} \frac{e^{-2(z-t)^2}}{\sqrt{\erfc(2t)}} \,dt, \qquad E:=(-\infty,0).
    \end{equation}
\end{thm}

Note that with \eqref{fE NH Hedge}, the kernel $\KK$ in \eqref{K complex} is given by
\begin{equation} \label{K NH hE}
 \KK(z,w)=\frac{4}{ \sqrt{\pi} }\, e^{2 z \bar{w}} \int_{-\infty}^0 \frac{ e^{-(z+\bar{w}-2u)^2} }{ \erfc(2u) } \,du.
\end{equation}
Again, \eqref{K NH hE} corresponds to the kernel of the complex Ginibre point process with soft/hard edge condition, see e.g.~\cite{MR3975882,MR4030288,forrester2010log}.

We now turn to the hard edge Ginibre ensemble.
This corresponds to the ensemble \eqref{Gibbs} with the potential
\begin{equation} \label{Q Ginibre HW}
	Q(\zeta)=
	\begin{cases}
		|\zeta|^2 &\text{if } |\zeta|\le\sqrt{2}\rho,
		\\
		\infty &\text{otherwise},
	\end{cases} \qquad \rho \in (0,1).
\end{equation}
By definition, the eigenvalues are confined in a disk $D_\rho:=\mathbb{D}(0,\sqrt{2}\rho)$. Thus the rescaled processes are again confined in $\mathbb{H}_-$.
In this case, the associated equilibrium measure is no longer absolutely continuous with respect to the area measure $dA$ and rather it is of the form
\begin{equation}
\frac{1}{2}\Delta Q \mathbf{1}_{D_\rho}\,dA + \frac{1-\rho^2}{\sqrt{2}\rho}\,ds,
\end{equation}
where $ds$ is the normalized arc-length measure on $\partial D_\rho$.
Therefore we choose the micro-scale at the edge-point $p=\sqrt{2}\rho$ as
\begin{equation}\label{microscale HW}
    \gamma_N = \frac{\sqrt{2}\rho}{N(1-\rho^2)}.
\end{equation}
We then obtain the following.

\begin{thm} \label{Thm_NH hW}\textbf{\textup{(Non-Hermitian hard edge scaling limit)}} Let $Q$ be the hard edge Ginibre potential \eqref{Q Ginibre HW}.
Then for $p = \pm \sqrt{2}\rho$, the $k$-point correlation function $R_{N,k}$ converges, as $N \to \infty$, locally uniformly to the limit
\begin{equation} \label{R Pfaffian HW}
R_k(z_1,\dots,z_k)=\prod_{j=1}^k (\bar{z}_j-z_j) \Pf \Big[
\begin{pmatrix}
\kappa(z_j,z_l) & \kappa(z_j,\bar{z}_l)
\smallskip
\\
\kappa(\bar{z}_j,z_l) & \kappa(\bar{z}_j,\bar{z}_l)
\end{pmatrix}
\Big]_{j,l=1}^k,
\end{equation}
where the pre-kernel $\kappa$ is of Wronskian form
 \begin{equation} \label{fE NH HWedge}
\kappa(z,w):=\int_{E} W(f_{w},f_{z})(u) \, du, \qquad        f_z(u) := \int_{0}^{u} t^{\frac{1}{2}}e^{2zt} \, dt,\qquad E:=(0,1).
    \end{equation}
\end{thm}

We mention that the function $f_z$ can be written in terms of the incomplete gamma function $\gamma(a,z)$ as
\begin{equation}
f_z(u)= \frac{1}{(-2z)^{\frac32}} \gamma(\tfrac32,-2uz).
\end{equation}

Compared to the limiting correlation functions of the form \eqref{R Pfaffian} with \eqref{kappa Wronskian}, in the hard edge scaling limit \eqref{R Pfaffian HW}, there are no Gaussian factors.
This comes from the micro-scale \eqref{microscale HW} used only for the hard edge ensemble among the models under consideration in the present work.
By a similar reason, with the choice of \eqref{fE NH HWedge}, it is natural to consider
\begin{equation} \label{K NH hW}
\KK(z,w):= 4\int_E f_z'(u) f_{\bar{w}}'(u)\,du= 4\int_E u \, e^{2u(z+\bar{w})}\,du
\end{equation}
rather than the one of the form \eqref{K complex}. Then this kernel \eqref{K NH hW} corresponds to the one obtained from the complex Ginibre ensemble with hard edge constraint, see \cite{seo2020edge}.

Concerning the equivalence of the truncated unitary and hard edge Ginibre ensemble \cite{seo2020edge,MR1748745}, we expect that the scaling limit \eqref{fE NH HWedge} coincides with the one obtained in the context of the truncated symplectic ensemble, see \cite{Lysychkin,BL} for the scaling limit away from the real edge.

\begin{figure}[h!]
	\begin{subfigure}{0.32\textwidth}
		\begin{center}
			\includegraphics[width=\textwidth]{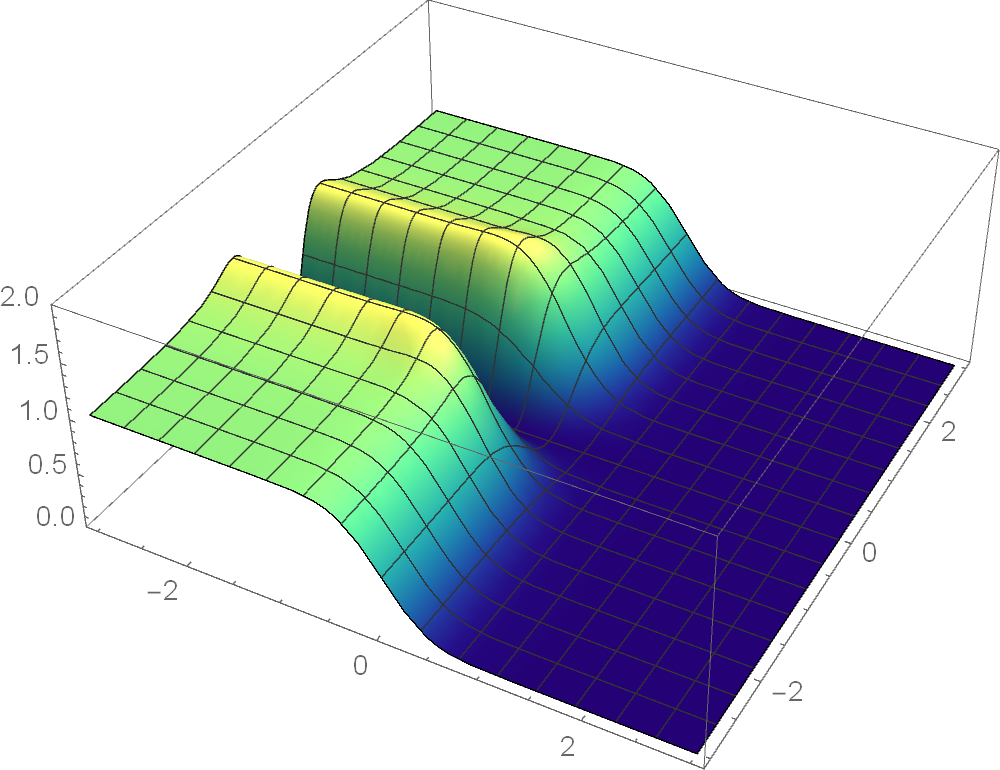}
		\end{center}
		\subcaption{free boundary}
	\end{subfigure}
	\begin{subfigure}{0.32\textwidth}
		\begin{center}
			\includegraphics[width=\textwidth]{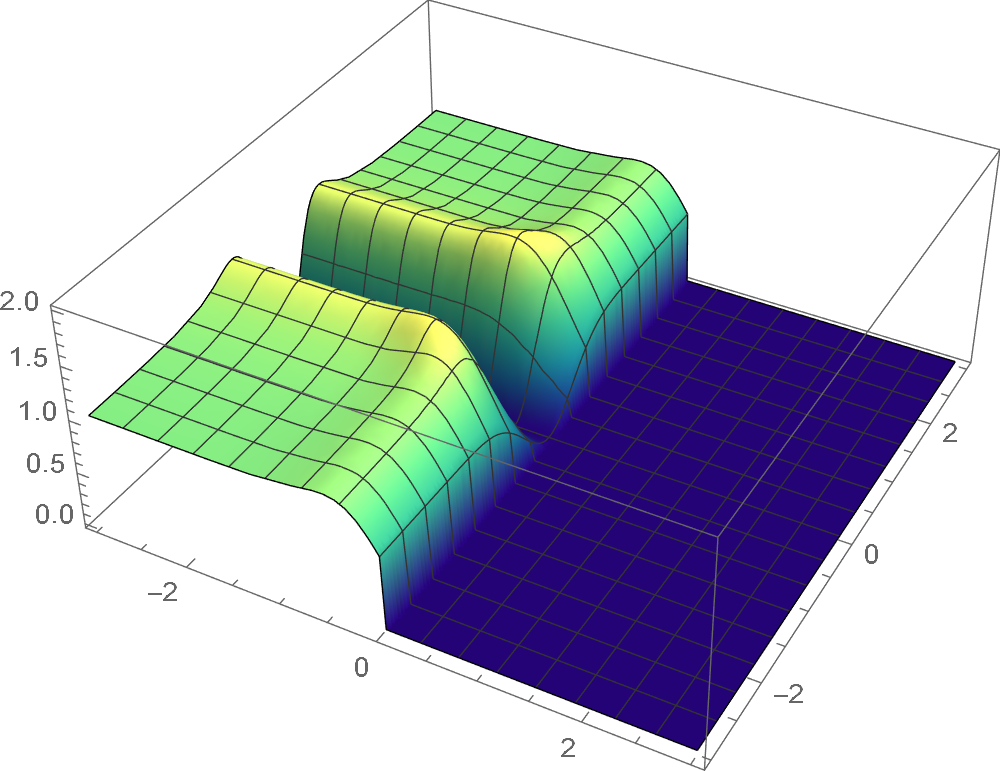}
		\end{center} \subcaption{soft/hard edge}
	\end{subfigure}
		\begin{subfigure}{0.32\textwidth}
		\begin{center}
			\includegraphics[width=\textwidth]{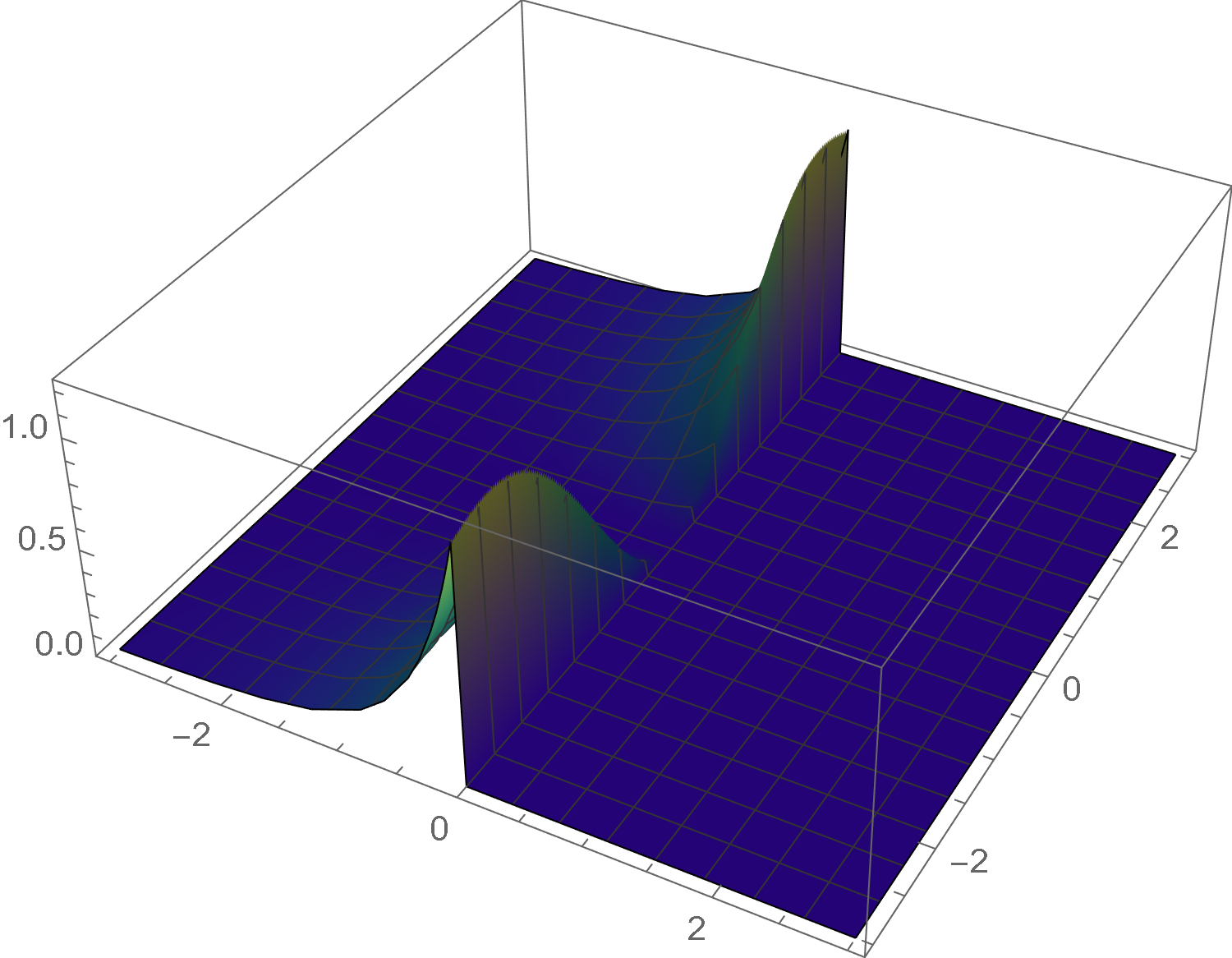}
		\end{center} \subcaption{hard edge}
	\end{subfigure}
	\caption{ The graphs of the one-point function $R$ in \eqref{R one pt f} with \eqref{fE NH bulk edge}, \eqref{fE NH Hedge}, and \eqref{fE NH HWedge} respectively. } \label{Fig_R edge FH}
\end{figure}

\begin{rmk*}[Numerics on subleading terms]
Beyond the limiting correlation functions, a natural question arising in the study of scaling limits is their rates of convergence as $N \to \infty$.
For instance, these were obtained in \cite{byun2021universal} for the symplectic elliptic Ginibre ensemble with fixed $\tau \in [0,1)$. (See also \cite{MR2208159} for the Hermitian counterparts.) In particular, it was shown that in the edge scaling limit, the convergence rate is of order $O(N^{-1/2})$. %(cf. in the bulk case, the speed is exponential).

The precise asymptotic expansions in the situations of Theorems~\ref{Thm_AH edge},~\ref{Thm_NH hE}, and ~\ref{Thm_NH hW} exceed the scope of this paper.
Nevertheless, we present some relevant numerical simulations.
In particular, from the numerics below, we observe that the rates of convergences are of order $O(N^{-1/3})$ in Theorem~\ref{Thm_AH edge} (cf. \cite{MR2208159}), of order $O(N^{-1/2})$ in Theorem~\ref{Thm_NH hE} and of order $O(N^{-3/4})$ in Theorem~\ref{Thm_NH hW}.

\begin{figure}[h!]
	\begin{subfigure}{0.32\textwidth}
		\begin{center}
			\includegraphics[width=\textwidth]{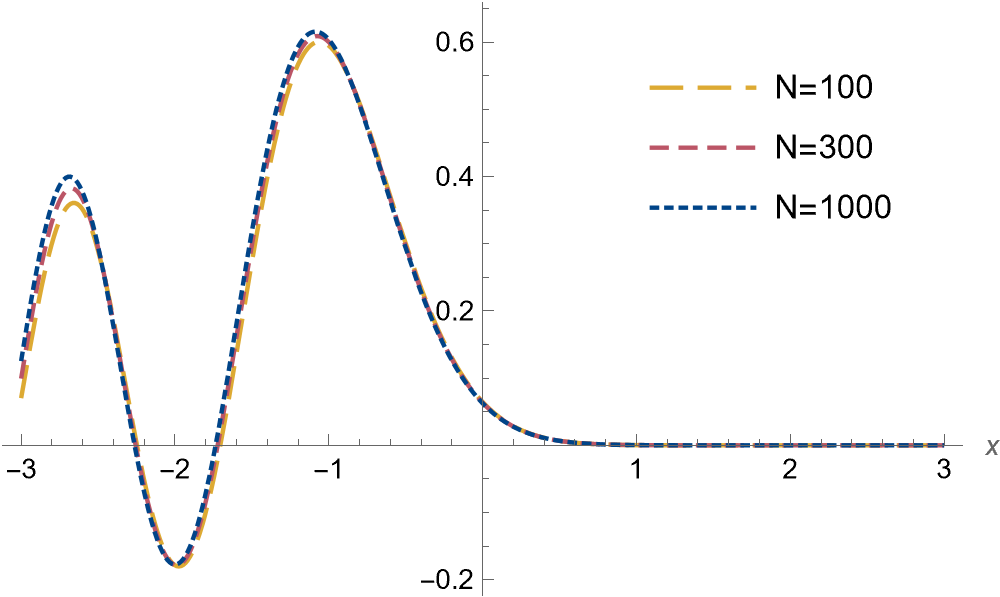}
		\end{center}
		\subcaption{almost-Hermitian edge, $y = 0.7$}
	\end{subfigure}
	\begin{subfigure}{0.32\textwidth}
		\begin{center}
			\includegraphics[width=\textwidth]{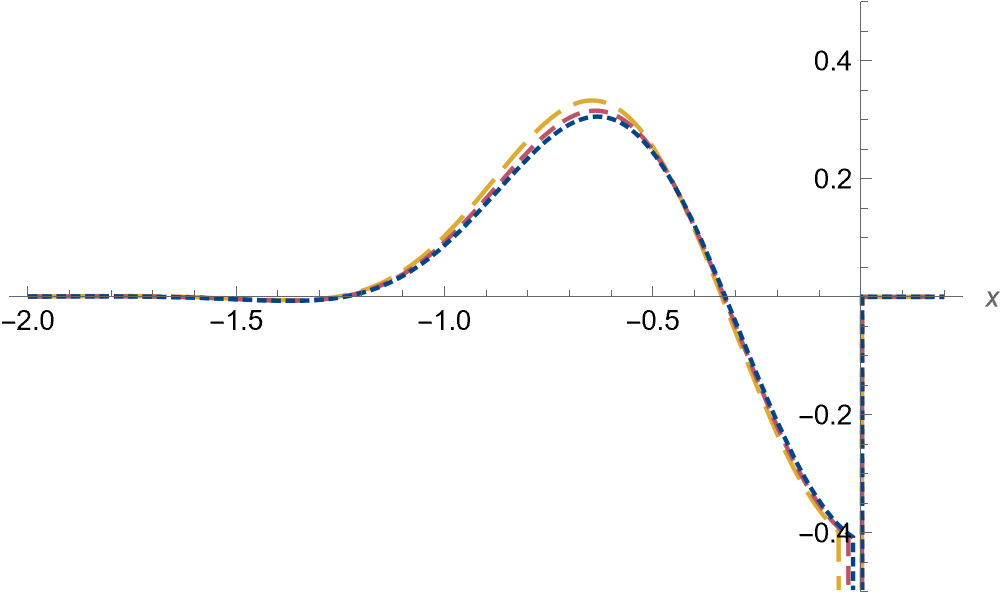}
		\end{center} \subcaption{soft/hard edge, $y = 1$}
	\end{subfigure}
		\begin{subfigure}{0.32\textwidth}
		\begin{center}
			\includegraphics[width=\textwidth]{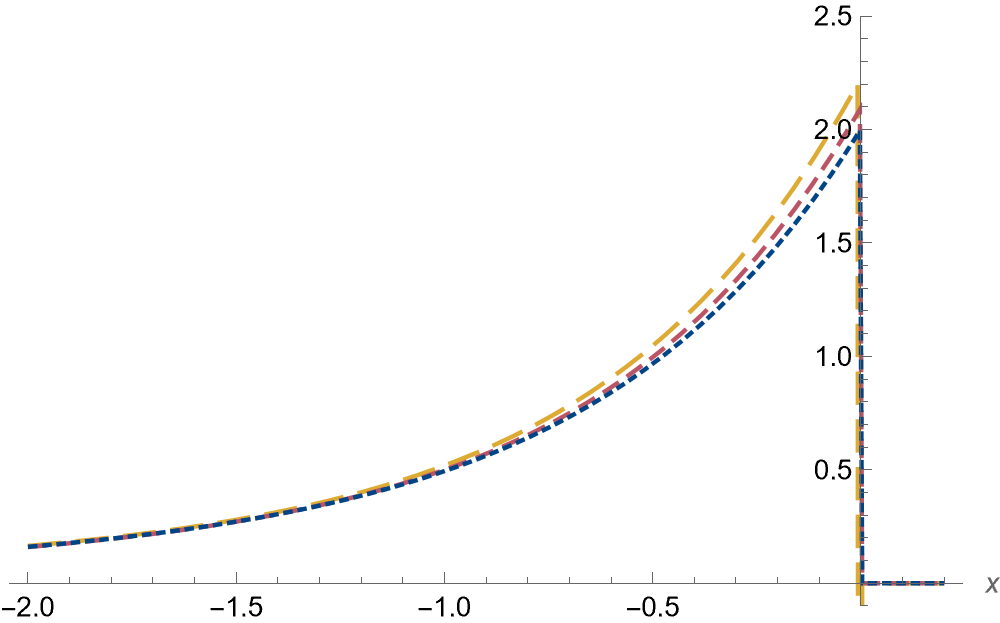}
		\end{center} \subcaption{hard edge, $y = 1$}
	\end{subfigure}

		\begin{subfigure}{0.32\textwidth}
		\begin{center}
			\includegraphics[width=\textwidth]{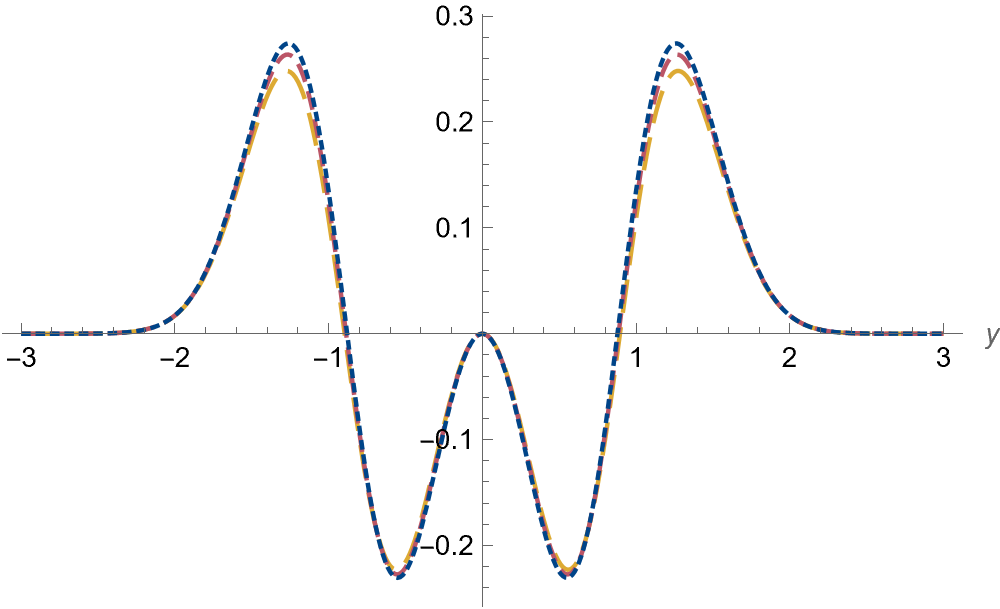}
		\end{center}
		\subcaption{almost-Hermitian edge, $x = 2$}
	\end{subfigure}
	\begin{subfigure}{0.32\textwidth}
		\begin{center}
			\includegraphics[width=\textwidth]{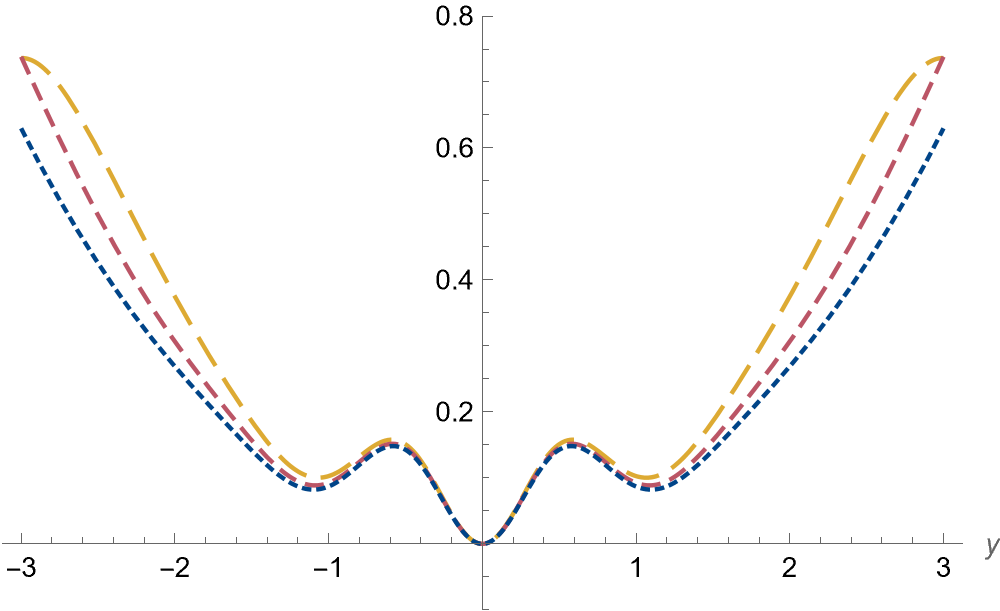}
		\end{center} \subcaption{soft/hard edge, $x = 1$}
	\end{subfigure}
		\begin{subfigure}{0.32\textwidth}
		\begin{center}
			\includegraphics[width=\textwidth]{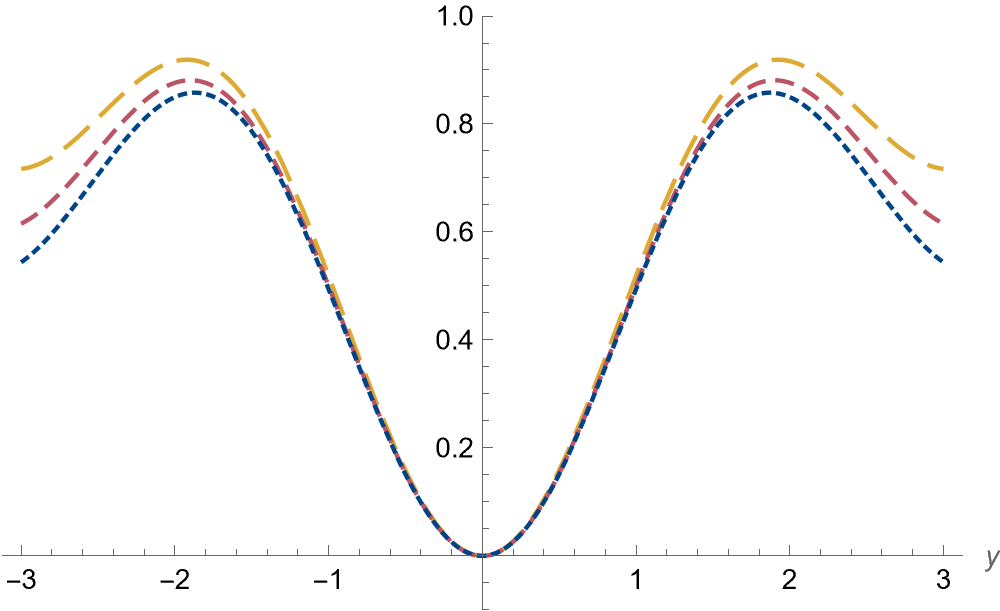}
		\end{center} \subcaption{hard edge, $x = 1$}
	\end{subfigure}
	\caption{ (A)--(C) display the graphs of $x \mapsto R_N^{(r)}(x+iy)$ for a given value of $y \in \R$, where $R_N^{(r)}(z):=N^r (R_N(z) - R(z))$.
	Here, $r=1/3$ for the almost-Hermitian edge (with $c=1$), $r=1/2$ for the soft/hard edge and $r=3/4$ for the hard edge cases respectively.
	(D)--(F) are the same figures for $y \mapsto R_N^{(r)}(x+iy)$. }
\end{figure}

\end{rmk*}

We end this section by giving a remark on universality.

\begin{rmk*}[Towards local universality]
Let $\upsilon(z,w):=e^{-z^2-w^2}\kappa(z,w). $
Then integration by parts gives rise to
\begin{align} \label{CDI limit}
\begin{split}
\pa_z\upsilon(z,w)=-2 \int_E \frac{\pa}{\pa z} f_z(u) \cdot f_w'(u)\,du+f_w(u) \frac{\pa}{\pa z} f_z(u)\Big|_{\pa E}.
\end{split}
\end{align}
For the free boundary cases, the first term on the right-hand side corresponds to the kernel of the complex ensembles, cf.~\eqref{K complex}.

The finite-$N$ version of this equation for the (elliptic) Ginibre potential was introduced in \cite{akemann2021scaling,byun2021universal} as a version of the Christoffel-Darboux formula.
Such a relation for some singular potentials has been investigated as well, which requires higher (or fractional) order differential operators, see \cite{MR2180006,osborn2004universal,MR4229527} for the Laguerre ensembles and  \cite{akemann2021scaling,ameur2018random,chau1998structure} for the Mittag-Leffler ensembles.

We also refer to \cite{MR1762659,MR1675356} for similar equations in the Hermitian random matrix theory.
Together with the local universality of the complex ensembles, this relation plays an important role in the study of local universality for symplectic ensembles \cite{MR2306224,MR2335245}.
In a similar spirit, we expect that the finite-$N$ version of the identity \eqref{CDI limit} together with the bulk/edge universality of the determinantal Coulomb gas \cite{MR2817648,hedenmalm2017planar} (cf.\cite{ameur2021szego}) provides key ingredients in universality problems for planar symplectic ensembles.

\end{rmk*}

\subsection*{Organisation of the paper.}
In Section~\ref{Section_Preliminaries}, we compile and summarise the relevant materials on the planar symplectic ensembles such as the skew-orthogonal polynomial representation of the pre-kernel.

Section~\ref{Section_AH} is devoted to the proofs of Theorems~\ref{Thm_AH bulk} and ~\ref{Thm_AH edge}.
In Subsection~\ref{Subsec_CDI}, we recall the Christoffel-Darboux formula in \cite{byun2021universal} and provide the general strategy of the proofs.

In Section~\ref{Section_NH heW}, we prove Theorems~\ref{Thm_NH hE} and \ref{Thm_NH hW}. The strategy of the proof of these theorems using the Laplace method is outlined in Subsection~\ref{Subsec_Laplace}.

\section{Preliminaries}\label{Section_Preliminaries}

By definition, the $k$-point correlation function $\bfR_{N,k}$ of the system \eqref{Gibbs} is given by
\begin{equation}\label{bfRNk def}
\bfR_{N,k}(\zeta_1,\cdots, \zeta_k) := \frac{N!}{(N-k)!} \int_{\C^{N-k}} \P _N(\zeta_1,\dots,\zeta_N) \prod_{j=k+1}^N \, dA(\zeta_j).
\end{equation}
It is well known \cite{MR1928853} that the ensemble \eqref{Gibbs} forms a Pfaffian point process.
In other words, there is a two-variable function $\bfkappa_N$, called the \emph{(skew) pre-kernel}, such that
\begin{equation} \label{bfR Pfa}
\bfR_{N,k}(\zeta_1,\cdots, \zeta_k) =\prod_{j=1}^{k} (\overline{\zeta}_j-\zeta_j)  \Pf \Big[
e^{ -NQ(\zeta_j)/2-NQ(\zeta_l)/2 }
\begin{pmatrix}
\bfkappa_N(\zeta_j,\zeta_l) & \bfkappa_N(\zeta_j,\bar{\zeta}_l)
\smallskip
\\
\bfkappa_N(\bar{\zeta}_j,\zeta_l) & \bfkappa_N(\bar{\zeta}_j,\bar{\zeta}_l)
\end{pmatrix}  \Big]_{ j,l=1,\cdots, k }.
\end{equation}

By a change of measures, it is easy to see that the correlation function $R_{N,k}$ of the rescaled process $\boldsymbol{z}$ in \eqref{rescaling} is given by
\begin{equation}
R_{N,k}(z_1,\cdots, z_k)=\gamma_N^{2k} \,  \bfR_{N,k}(\zeta_1,\cdots,\zeta_k).
\end{equation}
In particular, with the rescaled pre-kernel
\begin{align} \label{kappa rescaling}
\begin{split}
\kappa_N(z,w) :=\gamma_N^3 \bfkappa_N(\zeta,\eta),
\end{split}
\end{align}
we have
\begin{equation} \label{RNk rescaling}
R_{N,k}(z_1,\cdots, z_k) =\prod_{j=1}^{k} (\bar{z}_j-z_j)  \Pf \Big[ e^{ -\frac{N}{2} ( Q( p+\frac{z_j}{\sqrt{N \delta}} )+Q( p+ \frac{z_l}{\sqrt{N \delta}} ) ) } \begin{pmatrix}
\kappa_N(z_j,z_l) & \kappa_N(z_j,\bar{z}_l) \\
\kappa_N(\bar{z}_j,z_l) & \kappa_N(\bar{z}_j,\bar{z}_l)
\end{pmatrix}  \Big]_{ j,l=1,\cdots k }.
\end{equation}
We remark that different pre-kernels may give rise to the same correlation functions.
In particular, we call two pre-kernels $\kappa_N$ and $\widetilde{\kappa}_N$ \emph{equivalent} if there exists a sequence of unimodular functions $g_N:\C \to \C$ with $g_N(\overline{\zeta}) = 1/g_N(\zeta)$ such that $\widetilde{\kappa}_N(z, w) = g_N(z) g_N(w) \kappa_N(z, w)$.
In what follows, we also call $c_N(z, w) := g_N(z) g_N(w)$ a \emph{cocycle}.

The skew-symmetric form $\langle \cdot , \cdot  \rangle_s$ is given by
\begin{equation*}
\langle f, g \rangle_s := \int_{\C} \Big( f(\zeta) g(\bar{\zeta}) - g(\zeta) f(\bar{\zeta}) \Big) (\zeta - \bar{\zeta}) e^{-N Q(\zeta)} \,dA(\zeta).
\end{equation*}
Let $q_m$ be a family of monic polynomials of degree $m$ that satisfy the following skew-orthogonality conditions with skew-norms $r_k > 0$: for all $k, l \in \mathbb{N}$
\begin{equation}
\langle q_{2k}, q_{2l} \rangle_s = \langle q_{2k+1}, q_{2l+1} \rangle_s = 0, \qquad \langle q_{2k}, q_{2l+1} \rangle_s = -\langle q_{2l+1}, q_{2k} \rangle_s = r_k  \,\delta_{k, l},
\end{equation}
where $\delta_{k, l}$ is the Kronecker delta.
Then the pre-kernel $\bfkappa_N$ has a canonical representation in terms of the skew-orthogonal polynomials
\begin{equation}\label{bfkappaN skewOP}
\bfkappa_N(\zeta,\eta)=\sum_{k=0}^{N-1} \frac{q_{2k+1}(\zeta) q_{2k}(\eta) -q_{2k}(\zeta) q_{2k+1}(\eta)}{r_k}.
\end{equation}

We present some examples of skew-orthogonal polynomials that will be used in the following sections.

\begin{itemize}
    \item \textbf{Example 1. (Elliptic potential)}
For the elliptic potential $Q$ in \eqref{Q elliptic}, it was obtained by Kanzieper \cite{MR1928853} that the associated skew-orthogonal polynomials $q_k$ can be expressed in terms of the Hermite polynomial
\begin{equation}
H_k(z):=(-1)^k e^{z^2} \frac{d^k}{dz^k} e^{-z^2}=k! \sum_{m=0}^{ \lfloor k/2 \rfloor } \frac{(-1)^m}{ m! (k-2m)! } (2z)^{k-2m}.
\end{equation}
To be precise, we have
\begin{equation} \label{Hermite skew}
q_{2k+1}(\zeta) = \Big( \frac{\tau}{2N} \Big)^{k+\frac12} H_{2k+1}\Big( \sqrt{ \tfrac{N}{2\tau} } \zeta \Big), \qquad q_{2k}(\zeta) = \Big( \frac{2}{N} \Big)^k k! \sum_{l=0}^{k}  \frac{(\tau/2)^l}{(2l)!!} H_{2l}  \Big( \sqrt{ \tfrac{N}{2\tau} } \zeta \Big)
\end{equation}
and their skew-norms $r_k$ are given by
\begin{equation} \label{Hermite skew norm}
r_k = 2(1-\tau)^{3/2}(1+\tau)^{1/2}\frac{(2k+1)!}{N^{2k+2}}.
\end{equation}
This also follows from a more general method of constructing skew-orthogonal polynomials \cite[Theorem 3.1]{akemann2021skew}.
\smallskip
\item \textbf{Example 2. (Radially symmetric potential)}  Let us consider a general radially symmetric potential $Q(\zeta)=Q(|\zeta|)$ with $Q(\zeta) \gg 4 \log |\zeta|$ as $\zeta \to \infty$. We write
\begin{equation} \label{hk Qradial}
    h_k:=\int_\C |\zeta|^{2k} e^{-N Q(\zeta)} \, dA(\zeta).
\end{equation}
for the orthogonal norm.
Then it is easy to show that
\begin{equation} \label{skew op_rad}
	q_{2k+1}(\zeta)=\zeta^{2k+1}, \qquad
	q_{2k}(\zeta)=\zeta^{2k}+\sum_{l=0}^{k-1}  \zeta^{2l} \prod_{j=0}^{k-l-1} \frac{h_{2l+2j+2}  }{ h_{2l+2j+1} }, \qquad
	r_k=2h_{2k+1}
\end{equation}
forms a family of skew-orthogonal polynomials, see e.g.~\cite[p.7]{MR3066113} and \cite[Theorem 3.1]{akemann2021skew}.
\end{itemize}

\section{Scaling limits of the almost-Hermitian ensembles}\label{Section_AH}

In this section, we study the elliptic Ginibre ensemble in the almost-Hermitian regime and prove Theorems~\ref{Thm_AH bulk} and ~\ref{Thm_AH edge}.
In Subsection~\ref{Subsec_CDI}, we outline the strategy of our proofs based on the Christoffel-Darboux formula.
Subsections~\ref{Subsec_AH bulk} and ~\ref{Subsec_AH edge} are devoted to the study of the bulk scaling limit (Theorem~\ref{Thm_AH bulk}) and the edge scaling limit (Theorem~\ref{Thm_AH edge}) respectively.

\subsection{Strategy of the proof: the Christoffel-Darboux formula}\label{Subsec_CDI}

Combining \eqref{kappa rescaling}, \eqref{bfkappaN skewOP}, \eqref{Hermite skew}, \eqref{Hermite skew norm}, we have the canonical representation of the pre-kernel
\begin{align*}
\begin{split}
 \label{kappaN cano}
\kappa_N(z,w)&=  \sqrt{2}(1+\tau)  \sum_{k=0}^{N-1}  \frac{ (\tau/2)^{k+\frac12} }{(2k+1)!!}   H_{2k+1} \Big(  \sqrt{ \tfrac{N}{2\tau} } p+\sqrt{\tfrac{1-\tau^2}{\tau}}z \Big)
 \sum_{l=0}^k  \frac{(\tau/2)^l}{(2l)!!} H_{2l}  \Big(  \sqrt{ \tfrac{N}{2\tau} } p+\sqrt{\tfrac{1-\tau^2}{\tau}}w  \Big)
\\
&\quad - \sqrt{2}(1+\tau)  \sum_{k=0}^{N-1}  \frac{ (\tau/2)^{k+\frac12} }{(2k+1)!!}   H_{2k+1} \Big(  \sqrt{ \tfrac{N}{2\tau} } p+\sqrt{\tfrac{1-\tau^2}{\tau}}w \Big)
 \sum_{l=0}^k  \frac{(\tau/2)^l}{(2l)!!} H_{2l}  \Big(  \sqrt{ \tfrac{N}{2\tau} } p+\sqrt{\tfrac{1-\tau^2}{\tau}}z  \Big).
\end{split}
\end{align*}
Following \cite{byun2021universal}, let us introduce
\begin{equation} \label{transformed kernel}
\widehat{\kappa}_N(z,w) := \omega_N(z,w) \kappa_N(z,w),
\end{equation}
where
\begin{equation}
\omega_N(z,w) = \exp\Big[ \tau \, \Big( p \sqrt{\tfrac{N}{2(1-\tau^2)}}+z \Big)^2 +  \tau \, \Big( p \sqrt{\tfrac{N}{2(1-\tau^2)}}+w \Big)^2 - \Big( p\sqrt{\tfrac{N}{1-\tau^2}}+\sqrt{2}z \Big) \Big( p\sqrt{\tfrac{N}{1-\tau^2}}+\sqrt{2}w \Big) \Big].
\end{equation}
Note that by \eqref{Q elliptic} and \eqref{microscale Laplacian}, we have
\begin{equation}
e^{-\frac{N}{2} ( Q( p+\frac{z}{\sqrt{N \delta}} )+Q( p+ \frac{w}{\sqrt{N \delta}} ) )} \frac{1}{\omega_N(z,w)}
=e^{-\abs{z}^2 -\abs{w}^2 +2zw}  \, \frac{1}{c_N(z,w)},
\end{equation}
where the cocycle $c_N(z,w)$ is given by
\begin{equation}
c_N(z,w)=\exp\Big( -i\sqrt{2N\tfrac{1-\tau}{1+\tau}}\, p \, \im z - i\sqrt{2N\tfrac{1-\tau}{1+\tau}}\, p \, \im w +i\tau\im z^2 +i\tau \im w^2 \Big).
\end{equation}
Therefore as $N \to \infty$, we have the uniform limit
\begin{equation} \label{up to cocycle}
\lim_{N\to\infty} c_N(z,w) e^{-\frac{N}{2} ( Q( p+\frac{z}{\sqrt{N \delta}} )+Q( p+ \frac{w}{\sqrt{N \delta}} ) )} \kappa_N(z,w)
= e^{-\abs{z}^2 -\abs{w}^2 +2zw} \, \widehat{\kappa}(z,w),
\end{equation}
where $\widehat{\kappa}:=\lim_{N\to \infty} \widehat{\kappa}_N$.
Here the convergence is uniform on compact subsets of $\C$.

The key idea which allows us to perform the asymptotic analysis is the following version of the Christoffel-Darboux formula \cite[Proposition 1.1]{byun2021universal}.

\begin{lem} \textup{(Christoffel-Darboux formula for the skew-orthogonal Hermite polynomial kernel)} \label{Lem_CDI skew}
We have
\begin{equation}
 \partial_z \widehat{\kappa}_N(z,w)=2(z-w)\widehat{\kappa}_N(z,w) + \RN{1}_N(z,w)-\RN{2}_N(z,w),
\end{equation}
where
\begin{equation}
\RN{1}_N(z,w)=2\sqrt{1-\tau^2} \,  \omega_N(z,w) \sum_{k=0}^{2N-1}  \frac{ (\tau/2)^{k} }{k!}   H_{k}\Big( \sqrt{ \tfrac{N}{2\tau} } p+\sqrt{\tfrac{1-\tau^2}{\tau}}z \Big) H_{k}  \Big( \sqrt{ \tfrac{N}{2\tau} } p+\sqrt{\tfrac{1-\tau^2}{\tau}}w \Big)
\end{equation}
and
\begin{equation}
\RN{2}_N(z,w)=   2\sqrt{1-\tau^2}  \,\omega_N(z,w)
 \frac{ (\tau/2)^{N} }{(2N-1)!!}  H_{2N}  \Big( \sqrt{ \tfrac{N}{2\tau} } p+\sqrt{\tfrac{1-\tau^2}{\tau}}z \Big)  \sum_{l=0}^{N-1}  \frac{(\tau/2)^l}{(2l)!!} H_{2l}  \Big( \sqrt{ \tfrac{N}{2\tau} } p+\sqrt{\tfrac{1-\tau^2}{\tau}}w \Big) .
\end{equation}
\end{lem}

We now write
\begin{equation} \label{upsilon kappa}
\upsilon(z,w):=e^{-z^2-w^2}\kappa(z,w)=e^{-(z-w)^2} \widehat{\kappa}(z,w)
\end{equation}
and
\begin{equation}
\RN{1}(z,w):=\lim_{N\to\infty}\RN{1}_N(z,w), \qquad \RN{2}(z,w):=\lim_{N\to\infty}\RN{2}_N(z,w),
\end{equation}
where the convergence is uniform on compact subsets of $\C$ (we will show this later in the proof of Proposition~\ref{Prop_inhom terms AH bulk}).
Then by Lemma~\ref{Lem_CDI skew}, we have
\begin{equation}  \label{upsilon differential eq}
\frac{\pa}{\pa z} \upsilon(z,w)=e^{-(z-w)^2} \Big( \RN{1}(z,w)-\RN{2}(z,w) \Big).
\end{equation}

We analyse the large-$N$ limit of $\RN{1}_N$ and $\RN{2}_N$ using the expressions
\begin{equation}
\RN{1}_N(z,w) = \RN{1}_N(z_0,w) + \int_{z_0}^{z} \big( \pa_z \RN{1}_N \big)(t,w) \, dt, \qquad
\RN{2}_N(z,w) = \RN{2}_N(z,w_0) + \int_{w_0}^{w} \big( \pa_w \RN{2}_N \big)(z,t) \, dt.
\end{equation}
The following version of the Christoffel-Darboux formula for the kernel of the complex elliptic Ginibre ensemble was obtained by Lee and Riser \cite[Proposition 2.3]{MR3450566}, see also \cite[Section 3]{byun2021lemniscate} for more general identities of such kind.

\begin{lem} \textup{(Christoffel-Darboux formula for the orthogonal Hermite polynomial kernel)} \label{Lem_CDI orthogonal}
The function
\begin{equation}
S_N(\zeta, \eta) := \sum_{k=0}^{N-1} \frac{(\tau/2)^k}{k!} H_k(\zeta) H_k(\eta)
\end{equation}
satisfies
\begin{equation} \label{Derivative S_N}
\begin{split}
\pa_\zeta S_N(\zeta, \eta) &= \frac{2\tau}{1-\tau^2} (\eta-\tau\zeta) S_N(\zeta, \eta) + \frac{2}{1-\tau^2} \Big( \frac{\tau}{2} \Big)^{N} \frac{\tau H_{N}(\zeta) H_{N-1}(\eta) - H_{N-1}(\zeta) H_{N}(\eta)}{(N-1)!}.
\end{split}
\end{equation}
\end{lem}

Using this lemma, we have the following expressions.

\begin{lem} \label{Lem_RN12 derivatives}
We have
\begin{align}
\pa_z \RN{1}_N(z,w)&= \frac{ 4 }{ \sqrt{\tau}}  \frac{ (\tau/2)^{2N} }{(2N-1)!} \omega_N(z,w)  \Big( \tau H_{2N}(\zeta) H_{2N-1}(\eta) - H_{2N-1}(\zeta) H_{2N}(\eta) \Big)
\end{align}
and
\begin{equation}
\pa_w \Big[ e^{-(z-w)^2} \RN{2}_N(z,w) \Big] =   (\tau-1)\frac{4}{\sqrt{\tau}} \frac{  (\tau/2)^{2N} }{(2N-1)!} e^{-(z-w)^2}\omega_N(z,w) H_{2N}(\zeta) H_{2N-1}(\eta),
\end{equation}
where
\begin{equation} \label{Hermite arguments zeta eta}
\zeta=\sqrt{ \tfrac{N}{2\tau} } p+\sqrt{\tfrac{1-\tau^2}{\tau}}z, \qquad \eta=\sqrt{ \tfrac{N}{2\tau} } p+\sqrt{\tfrac{1-\tau^2}{\tau}}w.
\end{equation}
\end{lem}

In the following subsections, we derive the asymptotics of the inhomogeneous terms in Lemma~\ref{Lem_CDI skew}, see Propositions~\ref{Prop_inhom terms AH bulk} and ~\ref{Prop_inhom terms AH edge} below.
Then, by solving the resulting limiting differential equations for the pre-kernel, we complete the proof of Theorems~\ref{Thm_AH bulk} and ~\ref{Thm_AH edge}.

\subsection{Almost-Hermitian bulk scaling limit} \label{Subsec_AH bulk}

In this subsection, we consider the almost-Hermitian bulk scaling limit where $\tau$ is given by \eqref{tau AH bulk} and $p\in (-\sqrt{2}(1+\tau),\sqrt{2}(1+\tau))$.

We aim to show the following proposition.
\begin{prop} \label{Prop_inhom terms AH bulk}
We have
\begin{equation} \label{RN12 AH bulk}
\RN{1}(z,w) =\frac{2}{\sqrt{\pi}}  \, e^{(z-w)^2} \int_{I} e^{-t^2} \cos(2t(z-w)) \,dt, \qquad  \RN{2}(z,w)=0,
\end{equation}
where $I$ is given by \eqref{I tilde c}.
\end{prop}

We prove Theorem~\ref{Thm_AH bulk} using Proposition~\ref{Prop_inhom terms AH bulk}.

\begin{proof}[Proof of Theorem~\ref{Thm_AH bulk}]

Let us first show the alternative representation \eqref{kappa t.i.}.
Note that
\begin{align}
f_z'(u)=-\frac{1}{\pi} \int_I e^{-t^2/2} \cos(2t(u-z))\,dt, \quad I=(-\tilde{c},\tilde{c}).
\end{align}
Thus we have
\begin{align*}
&\quad W(f_{w},f_{z})(u) =  f_w(u) f_z'(u)-f_z(u) f_w'(u)
\\
&= \frac{1}{2\pi^2} \int_{I^2} e^{-t^2/2-s^2/2} \Big( \sin(2t(u-w)) \cos(2s(u-z))- \sin(2t(u-z)) \cos(2s(u-w)) \Big) \frac{1}{t} \,dt\,ds
\end{align*}
Notice here that
\begin{align*}
&\quad \sin(2t(u-w)) \cos(2s(u-z))=\frac{1}{4i}(e^{2it(u-w)}-e^{-2it(u-w)}) ( e^{2is(u-z)}+e^{-2is(u-z)}  )
\\
&= \frac{1}{4i} \Big( e^{2i(t+s)u} e^{-2itw-2is z}+e^{2i(t-s)u} e^{-2itw+2isz}-e^{2i(s-t)u}e^{2itw-2isz}-e^{-2i(t+s)u}e^{2itw+2isz} \Big).
\end{align*}
Since
\begin{equation*}
\delta(x)=\frac{1}{2\pi} \int_\R e^{iux}\,du = \frac{1}{\pi}  \int_\R e^{2iux}\,du
\end{equation*}
in the sense of distribution, we have
\begin{align*}
\int_\R \sin(2t(u-w)) \cos(2s(u-z))\,du&=\frac{\pi}{4i} \Big( \delta(t+s)+\delta(t-s) \Big)  \Big(e^{2it(z-w)}-e^{-2it(z-w)}\Big)
\\
&=\frac{\pi}{2} \Big( \delta(t+s)+\delta(t-s) \Big) \sin(2t(z-w)).
\end{align*}
Therefore we have
\begin{align*}
&\quad \int_I \int_\R e^{-t^2/2-s^2/2} \Big( \sin(2t(u-w)) \cos(2s(u-z))- \sin(2t(u-z)) \cos(2s(u-w)) \Big) \,du\,ds
\\
&= \pi \int_I \Big( \delta(t+s)+\delta(t-s) \Big) e^{-t^2/2-s^2/2} \sin(2t(z-w))\,ds = 2\pi \,e^{-t^2} \sin(2t(z-w)).
\end{align*}
By interchanging the integrals, this gives rise to
\begin{align*}
 \int_{ E } W(f_{w},f_{z})(u) \, du &=\frac{1}{\pi}\int_I e^{-t^2} \sin(2t(z-w))\frac{dt}{t}, \qquad E=\R,
\end{align*}
which leads to \eqref{kappa t.i.}.

Now it suffices to show that the function
\begin{equation}
\upsilon(z,w):=e^{-z^2-w^2}\kappa(z,w)= \frac{1}{\sqrt{\pi}}   \int_I e^{-u^2} \sin(2u(z-w) ) \frac{du}{u}
\end{equation}
is a unique anti-symmetric solution, which satisfies \eqref{upsilon differential eq} with \eqref{RN12 AH bulk}. This is immediate from the definition since
\begin{align*}
\frac{\pa}{\pa z} \upsilon(z,w) &= \frac{\pa}{\pa z} \frac{1}{\sqrt{\pi}}   \int_I e^{-u^2} \sin(2u(z-w) ) \frac{du}{u}
= \frac{2}{\sqrt{\pi}}   \int_I e^{-u^2} \cos(2u(z-w) )\,du.
\end{align*}
Now the proof is complete.
\end{proof}

The rest of this subsection is devoted to the proof of Proposition~\ref{Prop_inhom terms AH bulk}.
In the sequel, we shall consider the case $p \in [0,\sqrt{2}(1+\tau))$ as the other case follows along the same lines with slight modifications.
For the asymptotic analysis, we shall use the following strong asymptotics of the Hermite polynomials.

\begin{lem}\label{Lem_SA Hermite AH bulk}
As $N\to\infty$, we have for $p = 0$ (and $\zeta$ given in \eqref{Hermite arguments zeta eta})
\begin{equation}
H_{2N}(\zeta) \sim (-1)^N \frac{(2N)!}{N!} \cos(2cz) , \qquad H_{2N-1}(\zeta) \sim (-1)^{N+1} \frac{(2N)!}{\sqrt{4N-1}N!}  \sin(2cz) .
\end{equation}
For $p \in (0,\sqrt{2}(1+\tau))$, we have
\begin{align*}
\begin{split}
H_{2N}(\zeta) &\sim 2^{\frac54}\Big(\frac{4N}{e}\Big)^{N} (8-p^2)^{-\frac14} e^{\frac{N}{4}p^2+\frac{c^2p^2}{8}+\frac{pc}{\sqrt{2}}z}
\\
&\times \cos \Big[ \Big( 2 \arccos(\tfrac{p}{2\sqrt{2}})-\tfrac{p\sqrt{8-p^2}}{4} \Big)N +\tfrac{1}{2}\arccos(\tfrac{p}{2\sqrt{2}} )  - (\tfrac{c^2p}{8}+\tfrac{cz}{\sqrt{2}})\sqrt{8-p^2} - \tfrac{\pi}{4}  \Big]
\end{split}
\end{align*}
and
\begin{align*}
\begin{split}
H_{2N-1}(\zeta) &\sim 2^{\frac54}e^{-\frac12}\Big(\frac{4N}{e}\Big)^{N-\frac12} (8-p^2)^{-\frac14} e^{\frac{N}{4}p^2+\frac{c^2p^2}{8}+\frac{pc}{\sqrt{2}}z}
\\
&\times \cos \Big[ \Big( 2 \arccos(\tfrac{p}{2\sqrt{2}})-\tfrac{p\sqrt{8-p^2}}{4} \Big)N -\tfrac{1}{2}\arccos(\tfrac{p}{2\sqrt{2}} )  - (\tfrac{c^2p}{8}+\tfrac{cz}{\sqrt{2}})\sqrt{8-p^2} - \tfrac{\pi}{4}  \Big].
\end{split}
\end{align*}
\end{lem}
\begin{proof}
This immediately follows from the Mehler-Heine formula and Plancherel-Rotach asymptotics, see e.g.~\cite[Section~18.11(ii)]{olver2010nist} for $p = 0$ and   \cite[Corollary~4.1]{MR3274226}, cf. \cite[Appendix A]{AB21}, for $p \neq 0$.
\end{proof}

\begin{lem} \label{Lem_RNN12 derivative AH bulk}
As $N\to\infty$, we have
\begin{align}
\begin{split}
\pa_z \RN{1}_N(z,w) \sim  \frac{4}{\sqrt{\pi}} e^{-\tilde{c}^2} e^{(z-w)^2} \sin(2\tilde{c}(z-w))
\end{split}
\end{align}
and
\begin{align}
\begin{split}
&\quad \pa_w \Big[ e^{-(z-w)^2} \RN{2}_N(z,w) \Big] \sim \frac{c^2}{N} \frac{8\sqrt{2}}{\sqrt{\pi}}     (8-p^2)^{-\frac12}   e^{-c^2+ \frac{c^2p^2}{8}  }
\\
&\times \Big( \sin\Big[ 2\tilde{c}(z+w)+\tfrac{p\sqrt{8-p^2}}{4} (c^2+2N) -4N \arccos(\tfrac{p}{2\sqrt{2}} ) \Big]-\cos \Big[  2\tilde{c}(w-z)+\arccos(\tfrac{p}{2\sqrt{2}})  \Big] \Big).
\end{split}
\end{align}
\end{lem}
\begin{proof}
By \eqref{Hermite arguments zeta eta}, we have
\begin{equation*}
\zeta=\frac{p}{\sqrt{2}} \sqrt{N}+\Big( \frac{c^2p}{4\sqrt{2}}+c z \Big) \frac{1}{\sqrt{N}} +O(N^{-\frac32}), \qquad \frac{\zeta^2}{2}=\frac{p^2}{4}N+\Big( \frac{p^2c^2}{8}+\frac{pc}{\sqrt{2}}z \Big)+O(N^{-1}),
\end{equation*}
and similarly for $\eta$ in \eqref{Hermite arguments zeta eta}.
Furthermore we have
\begin{align*}
\begin{split}
\log \omega_N(z,w)&= -\frac{p^2}{2}N+(z-w)^2-\frac{cp}{\sqrt{2}}(z+w)-\frac{c^2p^2}{8}+O(N^{-1}), \qquad  2N \log \tau= -c^2+O(N^{-1}).
\end{split}
\end{align*}

We first consider the case $p=0$ (note that $\tilde{c} = c$). In this case, it follows from Lemma~\ref{Lem_SA Hermite AH bulk} that
\begin{align}
\begin{split}
H_{2N}(\zeta) H_{2N-1}(\eta) & \sim
- \frac{2^{4N-1} N! (N-1)! }{ \pi \sqrt{N} }   \cos(2cz) \sin(2cw).
\end{split}
\end{align}
Also notice that
\begin{equation}
\frac{ 4 }{ \sqrt{\tau}}  \frac{ (\tau/2)^{2N} }{(2N-1)!}   \frac{2^{4N-1} N! (N-1)! }{ \pi \sqrt{N} }   \sim  \frac{4}{\sqrt{\pi}} e^{-c^2}.
\end{equation}
Then by Lemma~\ref{Lem_RN12 derivatives}, we obtain
\begin{equation}
\pa_z \RN{1}_N(z,w) \sim  \frac{4}{\sqrt{\pi}} e^{(z-w)^2-c^2} \sin(2c(z-w))
\end{equation}
and
\begin{equation}
\pa_w \Big[ e^{-(z-w)^2} \RN{2}_N(z,w) \Big] \sim \frac{c^2}{N} \frac{8}{\sqrt{\pi}}  \,e^{-c^2} \cos(2cz) \sin(2cw) .
\end{equation}

Next, we prove the lemma for the case $p \in (0,\sqrt{2}(1+\tau))$.
By Lemma~\ref{Lem_SA Hermite AH bulk} and the elementary identity of trigonometric functions, we have
\begin{align*}
\begin{split}
&\quad H_{2N}(\zeta)  H_{2N-1}(\eta) \sim 2^{\frac32} e^{-\frac12}  \Big(\frac{4N}{e}\Big)^{2N-\frac12} (8-p^2)^{-\frac12} e^{\frac{N}{2}p^2+\frac{c^2p^2}{4}+\frac{cp}{\sqrt{2}}(z+w)}
\\
&\times \Big( \cos \Big[  2\tilde{c}(w-z)+\arccos(\tfrac{p}{2\sqrt{2}})  \Big]-\sin\Big[ \tfrac{\sqrt{8-p^2}}{4} \Big( 2\sqrt{2}c(z+w)+p(c^2+2N) \Big)-4N \arccos(\tfrac{p}{2\sqrt{2}} ) \Big] \Big).
\end{split}
\end{align*}
Then it follows from
\begin{align*}
&\quad  \cos \Big[  2\tilde{c}(w-z)+\arccos(\tfrac{p}{2\sqrt{2}})  \Big]- \cos \Big[  2\tilde{c}(z-w)+\arccos(\tfrac{p}{2\sqrt{2}})  \Big] =2^{-\frac12} (8-p^2)^{\frac12} \sin ( 2\tilde{c}(z-w) )
\end{align*}
that
\begin{align*}
\begin{split}
H_{2N}(\zeta)  H_{2N-1}(\eta)-H_{2N}(\eta)  H_{2N-1}(\zeta)
\sim 2\, e^{-\frac12}  \Big(\frac{4N}{e}\Big)^{2N-\frac12}  e^{\frac{N}{2}p^2+\frac{c^2p^2}{4}+\frac{cp}{\sqrt{2}}(z+w)}   \sin ( 2\tilde{c}(z-w) ) .
\end{split}
\end{align*}
Notice here that we have
\begin{align*}
&\quad \frac{ 4 }{ \sqrt{\tau}}  \frac{ (\tau/2)^{2N} }{(2N-1)!} \omega_N(z,w) \cdot  2 \, e^{-\frac12}  \Big(\frac{4N}{e}\Big)^{2N-\frac12}  e^{\frac{N}{2}p^2+\frac{c^2p^2}{4}+\frac{cp}{\sqrt{2}}(z+w)}
\sim \frac{4}{\sqrt{\pi}} \,  e^{-c^2+ \frac{c^2p^2}{8} +(z-w)^2 } .
\end{align*}
Now Lemma~\ref{Lem_RN12 derivatives} completes the proof.

\end{proof}

\begin{lem} \label{Lem_RN12 derivative AH bulk}
The function $\RN{1}$ given by \eqref{RN12 AH bulk} satisfies
\begin{equation}
\pa_z \RN{1}(z,w)=  \frac{4}{\sqrt{\pi}} e^{(z-w)^2-c^2} \sin(2c(z-w)).
\end{equation}
\end{lem}
\begin{proof}
By definition, we have
\begin{align*}
\pa_z \RN{1}(z,w) &=\frac{2}{\sqrt{\pi}}  \pa_z \Big[ e^{(z-w)^2} \int_I e^{-t^2} \cos(2t(z-w)) \,dt \Big]
\\
&=\frac{4}{\sqrt{\pi}} \,e^{(z-w)^2} \int_I e^{-t^2} \Big((z-w)\cos(2t(z-w))-t \sin(2t(z-w))\Big) \,dt.
\end{align*}
Now the lemma follows from integration by parts.
\end{proof}

We are now ready to prove Proposition~\ref{Prop_inhom terms AH bulk}.

\begin{proof}[Proof of Proposition~\ref{Prop_inhom terms AH bulk}]

By Lemmas~\ref{Lem_RNN12 derivative AH bulk} and ~\ref{Lem_RN12 derivative AH bulk}, all we need to show is the proposition for some initial values.

For the function $\RN{1}$, we choose the initial value $\RN{1}(\bar{w},w).$
As discussed in \cite{byun2021universal}, this corresponds to the microscopic density of the complex elliptic Ginibre ensemble.
Thus the initial value follows from for instance \cite{fyodorov1997almost,akemann2016universality,AB21}.

Now it remains to show that
$\RN{2}_N(z,w) \to \RN{2}(z,w)=0$ for some choice of $w$.
Here we choose $w=0$.

Let us write
\begin{equation} \label{RN2 TN}
e^{-z^2}\RN{2}_N(z,0) = 2 \sqrt{1-\tau^2} \, e^{-z^2} \omega_N(z,0) \frac{(\tau/2)^N}{(2N-1)!!} H_{2N}\Big( \sqrt{\tfrac{N}{2\tau}}p + \sqrt{\tfrac{1-\tau^2}{\tau}}z \Big) e^{\frac{N}{2(1+\tau)}p^2}  T_N(p),
\end{equation}
where
\begin{equation} \label{TN}
T_N(p) := e^{-\frac{N}{2(1+\tau)}p^2}  \sum_{l=0}^{N-1} \frac{(\tau/2)^l}{(2l)!!} H_{2l}\Big( \sqrt{\tfrac{N}{2\tau}}p \Big).
\end{equation}

We claim that $T_N(p) \sim 1/\sqrt{2}$.
Then it follows from straightforward computations using Lemma~\ref{Lem_SA Hermite AH bulk} that $\RN{2}_N(z,0)=O(1/\sqrt{N}).$

To analyse the term $T_N(p)$, let us write
\begin{equation} \label{eq:TNp Ansatz}
T_N(p) = T_N(0) + \int_{0}^{p} T'_N(t) \, dt.
\end{equation}
Due to the Hermite numbers $H_{2l}(0) = (-1)^l (2l)!/l!$, we have by Taylor expansion
\begin{equation} \label{TN 0}
T_N(0)=\sum_{l=0}^{N-1}\frac{\tau^l}{l!} \frac{(-1)^l (2l)!}{2^l (2l)!!}
= \frac{1}{\sqrt{1+\tau}} + O\Big( \frac{\tau^N}{\sqrt{N}} \Big)
= \frac{1}{\sqrt{2}} \Big( 1 + O\Big(\frac{1}{\sqrt{N}}\Big) \Big).
\end{equation}
By applying Lemmas~\ref{Lem_CDI orthogonal} and \ref{Lem_SA Hermite AH bulk},
\begin{equation}
\begin{split}
&\quad T'_N(t) = -\sqrt{\frac{N}{2\tau}} \frac{\tau^N}{1+\tau} \frac{1}{4^{N-1}(N-1)!} e^{-\frac{N}{2(1+\tau)}t^2} H_{2N-1}\Big( \sqrt{\tfrac{N}{2\tau}}t \Big)
\\
&\sim -\frac{\sqrt{N}}{\sqrt{2\pi}} (1-\tfrac{t^2}{8})^{-\frac{1}{4}} e^{-\frac{c^2}{2} + \frac{c^2}{16}t^2} \cos \Big[ \Big( 2 \arccos(\tfrac{t}{2\sqrt{2}})-\tfrac{t\sqrt{8-t^2}}{4} \Big)N -\tfrac{1}{2}\arccos(\tfrac{t}{2\sqrt{2}} )  - \tfrac{c^2t}{8}\sqrt{8-t^2} - \tfrac{\pi}{4}  \Big].
\end{split}
\end{equation}
Invoking this asymptotics, the resulting oscillatory integral in \eqref{eq:TNp Ansatz} can be analysed in the same way as \cite[Lemma 4.4]{byun2021real}.
To be more precise, it has the asymptotic form $\sqrt{N} \re \int_0^p a(t) e^{i N \phi(t)}\, dt$, where
\begin{equation}
a(t) = -\tfrac{1}{\sqrt{2\pi}} \Big(1-\tfrac{t^2}{8}\Big)^{-\frac{1}{4}} \exp\Big( -\tfrac{c^2}{2}+\tfrac{c^2}{16}t^2 -i\big(\tfrac{1}{2}\arccos(\tfrac{t}{2\sqrt{2}})+\tfrac{c^2t}{8}\sqrt{8-t^2}+\tfrac{\pi}{4} \big) \Big)
\end{equation}
and
\begin{equation}
\phi(t) = 2 \arccos(\tfrac{t}{2\sqrt{2}}) - \tfrac{t}{4}\sqrt{8-t^2}.
\end{equation}
Since $\phi'(t) = -\frac{1}{2}\sqrt{8-t^2} \neq 0$ as $p<2\sqrt{2}$, the Riemann–Lebesgue lemma yields
\begin{equation*}
\re \int_0^p a(t) e^{i N \phi(t)} \, dt = O\Big( \frac{1}{N} \Big),
\end{equation*}
see \cite[Lemma 4.4]{byun2021real} for further details.
This completes the proof.

\end{proof}

\subsection{Almost-Hermitian edge scaling limit} \label{Subsec_AH edge}

In this subsection, we consider the almost-Hermitian edge scaling limit where $\tau$ is given by \eqref{tau AH edge} and $p = \pm \sqrt{2}(1+\tau)$.
As in the previous subsection, we need to show the following.
\begin{prop} \label{Prop_inhom terms AH edge}
We have
\begin{equation} \label{RN1 AH edge}
 \RN{1}(z,w) = 8 \sqrt{\pi} \,c^2\, e^{(z-w)^2} \int_{-\infty}^0  e^{ c^3(z+w-2u)+\frac{c^6}{6} } \Ai\Big(2c(z-u)+\frac{c^4}{4}\Big) \Ai\Big(2c(w-u)+\frac{c^4}{4}\Big)\,du
\end{equation}
and
\begin{equation} \label{RN2 AH edge}
\RN{2}(z,w) = 4\sqrt{\pi} \,c^2\, e^{(z-w)^2}  e^{c^3 z} \Ai\Big(2cz+\frac{c^4}{4}\Big) \int_{-\infty}^0 e^{ c^3(w-u)+\frac{c^6}{6} }  \Ai\Big(2c(w-u)+\frac{c^4}{4}\Big)\,du.
\end{equation}
\end{prop}
We remark that by \eqref{Airy asym}, it is easy to see that
\begin{equation}
\lim_{c\to\infty} \RN{1}(z,w)=\erfc(z+w), \qquad  \lim_{c\to\infty} \RN{2}(z,w)= \tfrac{1}{\sqrt{2}}e^{(z-w)^2-2z^2} \erfc(\sqrt{2}w).
\end{equation}
These terms appear in the limiting differential equations for the non-Hermitian regime, see \cite{akemann2021scaling,byun2021universal}
.

We prove Theorem~\ref{Thm_AH edge} using Proposition~\ref{Prop_inhom terms AH edge}.

\begin{proof}[Proof of Theorem~\ref{Thm_AH edge}]
Let us write
\begin{equation}
\upsilon(z,w):=\sqrt{\pi}  \int_{-\infty}^0 W(f_{w},f_{z})(u) \, du=\sqrt{\pi}  \int_{-\infty}^0 \Big(f_{w}(u)f'_{z}(u)-f'_{w}(u)f_{z}(u)\Big) \, du,
\end{equation}
where $f_z$ and $E$ are given by \eqref{fE AH edge}.
Then similarly as above, it suffices to show that the function $\upsilon$  is a unique anti-symmetric solution, which satisfies \eqref{upsilon differential eq} with \eqref{RN1 AH edge} and \eqref{RN2 AH edge}.

To see this, note that
\begin{equation}
f'_{z}(u)=2c \, \exp\Big( c^3(z-u)+\frac{c^6}{12} \Big) \Ai\Big(2c(z-u)+\frac{c^4}{4}\Big).
\end{equation}
This gives
\begin{equation}
\frac{\pa}{\pa z} f'_{z}(u)=-f''_{z}(u).
\end{equation}

Notice that since
\begin{align*}
\frac{\pa}{\pa z} f_{z}(u) &=2c  \int_{0}^{u} \frac{\pa }{\pa z} \Big[\exp\Big( c^3(z-t)+\frac{c^6}{12} \Big) \Ai\Big(2c(z-t)+\frac{c^4}{4}\Big)\Big]\,dt
\\
&= -2c \int_{0}^{u} \frac{\pa }{\pa t} \Big[\exp\Big( c^3(z-t)+\frac{c^6}{12} \Big) \Ai\Big(2c(z-t)+\frac{c^4}{4}\Big)\Big]\,dt,
\end{align*}
we have
\begin{equation}
\frac{\pa}{\pa z} f_{z}(u)= -f'_{z}(u)+f'_{z}(0).
\end{equation}
Therefore we obtain
\begin{align*}
\frac{\pa}{\pa z} \upsilon(z,w)&=\sqrt{\pi}  \int_{-\infty}^0 \frac{\pa}{\pa z}\Big(f_{w}(u)f'_{z}(u)-f'_{w}(u)f_{z}(u)\Big) \, du
\\
&= \sqrt{\pi}  \int_{-\infty}^0 \Big(-f_{w}(u)f''_{z}(u)+f'_{w}(u)f'_{z}(u)-f'_{w}(u) f_{z}'(0)\Big) \, du
\\
&=2 \sqrt{\pi} \int_{-\infty}^0 f'_{z}(u)f'_{w}(u)\,du - \sqrt{\pi} \,\Big( f_{z}'(u)+f_{z}'(0) \Big) f_{w}(u) \Big|_{-\infty}^0,
\end{align*}
which completes the proof.
\end{proof}

The rest of this subsection is devoted to the proof of Proposition~\ref{Prop_inhom terms AH edge}.

For the asymptotic analysis, we shall use the following (critical) strong asymptotics of the Hermite polynomials.

\begin{lem} \label{Lem_SA Hermite AH edge}
As $N\to\infty$, we have (with $\zeta$ again given by \eqref{Hermite arguments zeta eta})
\begin{equation}
H_{2N}(\zeta) \sim (2\pi)^{\frac14} 2^{ N } \sqrt{ (2N)!} (2N)^{-\frac{1}{12}  }
	 e^{ \frac12 \zeta^2 } \Ai\Big( 2c z+\frac{c^4}{4} \Big)
\end{equation}
and
\begin{equation}
H_{2N-1}(\zeta) \sim (2\pi)^{\frac14} 2^{ N-\frac12 } \sqrt{(2N-1)!} (2N)^{-\frac{1}{12}  }
	 e^{ \frac12 \zeta^2 } \Ai\Big( 2c z+\frac{c^4}{4} \Big).
\end{equation}
\end{lem}
\begin{proof}
This immediately follows from the critical Plancherel-Rotach estimate
\begin{equation} \label{PR critical}
	H_N(\sqrt{2N}+\tfrac{z}{\sqrt{2} N^{1/6} }  ) \sim  (2\pi)^{\frac14} 2^{ \frac{N}{2} } \sqrt{N!} N^{-\frac{1}{12}  }
	 e^{ \frac12 (\sqrt{2N}+\frac{z}{\sqrt{2}N^{1/6} })^2 } \Ai(z),
\end{equation}
see e.g.~\cite{MR0000077}.
\end{proof}

\begin{lem} \label{Lem_RNN12 derivative AH edge}
As $N\to\infty$, we have
\begin{equation}
\pa_z \RN{1}_N(z,w)=-4\sqrt{\pi}c^2 e^{ c^3(z+w)+\frac{c^6}{6}+(z-w)^2 } \Ai\Big( 2c z+\frac{c^4}{4} \Big)  \Ai\Big( 2c w+\frac{c^4}{4} \Big) +O(N^{-\frac16})
\end{equation}
and
\begin{equation}
\pa_w \Big[ e^{-(z-w)^2} \RN{2}_N(z,w) \Big]= -4\sqrt{\pi}c^2 e^{ c^3(z+w)+\frac{c^6}{6} } \Ai\Big( 2c z+\frac{c^4}{4} \Big)  \Ai\Big( 2c w+\frac{c^4}{4} \Big) +O(N^{-\frac16}).
\end{equation}
\end{lem}
\begin{proof}

First note that for $\zeta$ given in \eqref{Hermite arguments zeta eta} we have
\begin{equation}
\zeta= 2\sqrt{N}+ \frac{ c^4/4+2c z }{ \sqrt{2}\,(2N)^{1/6}  } +O(N^{-\frac12}), \qquad  \frac{\zeta^2}{2}=2N+( 2^{-\frac53}c^4+2^{\frac43}c z )N^{\frac13}+\frac{c^6}{4}+\frac{c^3}{2}z+O(N^{-\frac16}).
\end{equation}
We also have
\begin{align}
\begin{split}
\log \omega_N(z,w)&= -4N +2^{\frac23}c^2 N^{\frac23}-2^{\frac43}\,c (z+w) N^{\frac13} +(z-w)^2+\frac{c^3}{2}(z+w)+O(N^{-\frac13})
\end{split}
\end{align}
and
\begin{equation}
2N \log \tau= -2^{\frac23} c^2 N^{\frac23}-2^{-\frac23}c^4 N^{\frac13}-\frac{c^6}{3}+O(N^{-\frac13}).
\end{equation}
These give rise to
\begin{align}
\begin{split}
\log \omega_N(z,w)+\frac{\zeta^2+\eta^2}{2}+2N \log \tau&= c^3(z+w)+\frac{c^6}{6}+(z-w)^2+O(N^{-\frac16}).
\end{split}
\end{align}

Note that by Lemma~\ref{Lem_SA Hermite AH edge}, we have
\begin{equation}
H_{2N}(\zeta) H_{2N-1}(\eta) \sim (2\pi)^{\frac12} 2^{ 2N-\frac12 } \sqrt{ (2N)! (2N-1)!} (2N)^{-\frac{1}{6}  }
 e^{ \frac12 (\zeta^2+\eta^2) } \Ai\Big( 2c z+\frac{c^4}{4} \Big)  \Ai\Big( 2c w+\frac{c^4}{4} \Big).
\end{equation}
Also notice that
\begin{equation}
(2\pi)^{\frac12} 2^{2N-\frac12} \sqrt{ (2N)! (2N-1)! } (2N)^{-\frac16}  \frac{ 4 }{ \sqrt{\tau}}  \frac{ 2^{-2N} }{(2N-1)!} (\tau-1) =-4\sqrt{\pi}c^2+O(N^{-\frac13}).
\end{equation}
Combining all of the above with Lemma~\ref{Lem_RN12 derivatives}, we obtain
\begin{equation}
\pa_z \RN{1}_N(z,w)= -4\sqrt{\pi}c^2 e^{ c^3(z+w)+\frac{c^6}{6}+(z-w)^2 } \Ai\Big( 2c z+\frac{c^4}{4} \Big)  \Ai\Big( 2c w+\frac{c^4}{4} \Big) +O(N^{-\frac16})
\end{equation}
and
\begin{equation}
\pa_w \Big[ e^{-(z-w)^2} \RN{2}_N(z,w) \Big]  = -4\sqrt{\pi}c^2 e^{ c^3(z+w)+\frac{c^6}{6} } \Ai\Big( 2c z+\frac{c^4}{4} \Big)  \Ai\Big( 2c w+\frac{c^4}{4} \Big) +O(N^{-\frac16}).
\end{equation}
This completes the proof.
\end{proof}

\begin{lem} \label{Lem_RN12 derivative AH edge}
The functions $\RN{1}$ and $\RN{2}$ given by \eqref{RN1 AH edge} and \eqref{RN2 AH edge} satisfy
\begin{equation}
\frac{\pa}{\pa z} \RN{1}(z,w)=-\sqrt{\pi} \,e^{(z-w)^2}\,  f_{z}'(0) f_{w}'(0), \qquad \frac{\pa}{\pa w} \Big[ e^{-(z-w)^2} \RN{2}(z,w) \Big] =-\sqrt{\pi} \,  f_{z}'(0) f_{w}'(0).
\end{equation}
\end{lem}

\begin{proof}
The second identity is obvious since
\begin{align*}
\pa_w \Big[ e^{-(z-w)^2} \RN{2}(z,w) \Big] &=-\sqrt{\pi} \, f_{z}'(0) \frac{\pa}{\pa w} f_{w}(-\infty)
\\
&= -\sqrt{\pi} \, f_{z}'(0) \Big( f'_{w}(0)-f'_{w}(-\infty) \Big)=-\sqrt{\pi} \,  f_{z}'(0) f_{w}'(0).
\end{align*}
For the first one, note that
\begin{align*}
\frac{\pa}{\pa z}\RN{1}(z,w) &= 2 \sqrt{\pi} \,\frac{\pa}{\pa z} \Big[ e^{(z-w)^2} \int_{-\infty}^0 f'_{z}(u)f'_{w}(u)\,du \Big]
\\
&=2\sqrt{\pi}\, e^{(z-w)^2}  \Big[  2(z-w)  \int_{-\infty}^0 f'_{z}(u)f'_{w}(u)\,du+  \int_{-\infty}^0 \frac{\pa}{\pa z} f'_{z}(u)f'_{w}(u)\,du \Big]
\\
&=2\sqrt{\pi}\, e^{(z-w)^2}  \Big[  2(z-w)  \int_{-\infty}^0 f'_{z}(u)f'_{w}(u)\,du - \int_{-\infty}^0  f''_{z}(u)f'_{w}(u)\,du \Big].
\end{align*}
Thus we have
\begin{align*}
\frac{\pa}{\pa z} \RN{1}(z,w)\Big|_{w=z}&=-2\sqrt{\pi}\,  \int_{-\infty}^0  f''_{z}(u)f'_{z}(u)\,du
=-2\sqrt{\pi}\,  f'_{z}(0)^2 -\frac{\pa}{\pa z} \RN{1}(z,w)\Big|_{w=z}.
\end{align*}
Therefore all we need to show is
\begin{align}
\begin{split}
\frac{\pa^2}{\pa z \pa w} \RN{1}(z,w)&=2\sqrt{\pi}(z-w) e^{(z-w)^2}   f_{z}'(0) f_{w}'(0)+\sqrt{\pi} \,e^{(z-w)^2}\,  f_{z}'(0) f_{w}''(0)
\\
&=-2\sqrt{\pi}(z-w) e^{(z-w)^2}   f_{z}'(0) f_{w}'(0)+\sqrt{\pi} \,e^{(z-w)^2}\,  f_{z}''(0) f_{w}'(0)
\\
&=\frac{\sqrt{\pi}}{2} e^{(z-w)^2} \Big( f_{z}'(0) f_{w}''(0)  + f_{z}''(0) f_{w}'(0)\Big).
\end{split}
\end{align}
This follows from straightforward computations using integration by parts.
\end{proof}

We now prove Proposition~\ref{Prop_inhom terms AH edge}.

\begin{proof}[Proof of Proposition~\ref{Prop_inhom terms AH edge}]

By Lemmas~\ref{Lem_RNN12 derivative AH edge} and ~\ref{Lem_RN12 derivative AH edge}, it suffices to show the proposition for some initial values.

As above, for the function $\RN{1}$, we choose the initial value $z=\bar{w}$. Then the value of $\RN{1}(\bar{w},w)$ follows from  \cite{bender2010edge,akemann2010interpolation,AB21}.

Therefore again, it remains to show $\RN{2}_N(z,w) \to \RN{2}(z,w)$ for some $w$.
For $w=0$, we shall show that
\begin{equation} \label{RN2 0 edge}
e^{-z^2}\RN{2}_N(z,0) \to 4\sqrt{\pi}c^2 \Ai\Big(2cz+\frac{c^4}{4}\Big) \int_{-\infty}^{0} e^{c^3(z-u)+\frac{c^6}{6}} \Ai\Big(-2cu+\frac{c^4}{4}\Big) \, du.
\end{equation}
For this purpose, we again use the expression \eqref{RN2 TN}.
After inserting \eqref{tau AH edge}, the limit of the prefactor is easy to compute and for the Hermite polynomials in $z$ we use the asymptotics from Lemma~\ref{Lem_SA Hermite AH edge}.
Only the sum $T_N(p_N)$ with $p_N=\sqrt{2}(1+\tau)$ remains for which we use the following ansatz:
\begin{equation} \label{eq:TNp Ansatz edge}
T_N(p_N) = T_N(0) + \int_{0}^{r_N} T'_N(t) \, dt + \int_{r_N}^{p_N} T'_N(t) \, dt, \qquad r_N := p_N - (2N)^{-\frac23+\alpha},
\end{equation}
where $0 < \alpha < \frac{4}{15}$.
Note that by Lemma~\ref{Lem_CDI orthogonal}, we have
\begin{equation}
T'_N(t)= -\sqrt{\frac{N}{2\tau}} \frac{\tau^N}{1+\tau} \frac{1}{4^{N-1}(N-1)!} e^{-\frac{N}{2(1+\tau)}t^2} H_{2N-1}\Big( \sqrt{\tfrac{N}{2\tau}}t \Big).
\end{equation}
Recall also that by \eqref{TN 0}, we have $\lim_{N\to\infty} T_N(0) = \frac{1}{\sqrt{2}}$.

In the bulk $t \in [0, r_N]$ we use the Hermite asymptotics from the oscillatory regime
\begin{equation}
H_{2N-1}\Big( \sqrt{\tfrac{N}{2\tau}} t \Big) = \Big( 4-\tfrac{t^2}{2} \Big)^{-\frac14} 2^{2N} N^{N-\frac12} \exp\Big( -N+\tfrac{N}{4}t^2+\tfrac{c^2}{8}t^2(2N)^{\frac23}+\tfrac{c^4}{8}t^2(2N)^{\frac13}+\tfrac{c^6}{8}t^2 \Big) O(1)
\end{equation}
where the $O(1)$-term is uniform because $\alpha > 0$, cf. \cite[Eq. (3.3), (3.5)]{skovgaard1959asymptotic}.
Hence the derivative has the asymptotic form
\begin{equation}
T'_N(t) = -\sqrt{N} \Big( 1-\tfrac{t^2}{8} \Big)^{-\frac14} \exp\Big( -\tfrac{c^2}{2}(1-\tfrac{t^2}{8})(2N)^{\frac23} -\tfrac{c^4}{4}(1-\tfrac{3t^2}{8})(2N)^{\frac13} -\tfrac{c^6}{8}(\tfrac{4}{3}-\tfrac{7t^2}{8}) \Big) O(1).
\end{equation}
Thus we conclude that there exists a constant $C > 0$ such that for all sufficiently large $N$ and all $t \in [0, r_N]$ the following inequality holds
\begin{equation}
\abs{T'_N(t)} \leq C \frac{\sqrt{N}}{\sqrt{c}} (2N)^{\frac{1}{12}} \exp\Big( -\frac{c^2}{2\sqrt{2}}N^\alpha+\frac{c^6}{12} \Big).
\end{equation}
Since $\alpha > 0$ we obtain $\int_{0}^{r_N} T'_N(t) dt \to 0$ for $N \to \infty$.

Near the edge we set $t=\sqrt{2}(1+\tau)+(2N)^{-\frac23}s$.
Then by Lemma~\ref{Lem_SA Hermite AH edge} we have
\begin{equation}
T'_N(t) = -2^{-\frac13} N^{\frac23} \exp\Big( \tfrac{c^6}{12}+\tfrac{c^2}{2\sqrt{2}}s \Big) \Big[ \Ai\big(\tfrac{s}{\sqrt{2}}+\tfrac{c^4}{4}\big) + O(N^{-\frac23+\frac{5}{2}\alpha}) \Big].
\end{equation}
and the error term is again uniform because $\alpha < \frac{1}{3}$, cf.\ \cite[Eq. (3.11), (3.13)]{skovgaard1959asymptotic}.
Thus we obtain that
\begin{equation}
\begin{split}
\int_{r_N}^{p_N} T'_N(t) \, dt
&= (2N)^{-\frac23} \int_{-(2N)^\alpha}^{0} T'_N\big( p_N+(2N)^{-\frac23}s \big) \, ds \\
&\sim -\frac{1}{2} e^{\frac{c^6}{12}} \int_{-\infty}^{0} e^{\frac{c^2}{2\sqrt{2}}s} \Ai\Big( \frac{s}{\sqrt{2}}+\frac{c^4}{4} \Big) \, ds
= -\sqrt{2} \, c \, e^{\frac{c^6}{12}} \int_{0}^{\infty} e^{-c^3 u} \Ai\Big( -2cu+\frac{c^4}{4} \Big) \, du.
\end{split}
\end{equation}
Also note that the Airy-function has the following Laplace transform \cite[Eq. (9.10.13)]{olver2010nist}
\begin{equation}
\int_{-\infty}^{\infty} e^{a v} \Ai(v) dv = e^{\frac{a^3}{3}}, \qquad \re a > 0.
\end{equation}
Therefore we can rewrite the initial value $1/\sqrt{2}$ as an integral over $e^{-c^3 u} \Ai(2cu+\frac{c^4}{4})$ with $u\in(-\infty,\infty)$.
Combining all of the above with \eqref{eq:TNp Ansatz edge}, we arrive at
\begin{equation}
T_N(p_N) \sim \sqrt{2} c \, e^{2N-\frac{c^2}{2}(2N)^{\frac{2}{3}}+\frac{c^6}{12}} \int_{-\infty}^{0} e^{-c^3 u} \Ai\Big(-2cu+\frac{c^4}{4}\Big) \, du.
\end{equation}
Now, it follows from the straightforward computations of the pre-factors that the claimed formula \eqref{RN2 0 edge} holds.

\end{proof}

\section{Scaling limits of the soft/hard and hard edge ensembles}\label{Section_NH heW}

In this section, we study the Ginibre ensemble with boundary confinements and prove Theorems~\ref{Thm_NH hE} and ~\ref{Thm_NH hW}.
In the first subsection, we summarise the strategy of our proofs using the Laplace method.
In Subsections~\ref{Subsec_SH edge} and ~\ref{Subsec_H edge}, we derive the boundary scaling limits of the Ginibre ensemble with soft/hard edge (Theorem~\ref{Thm_NH hE}) and hard edge (Theorem~\ref{Thm_NH hW}) constraint respectively.

\subsection{Strategy of the proof: the Laplace method}\label{Subsec_Laplace}

We consider the potential $Q$ of the form \eqref{Q Ginibre HE} and \eqref{Q Ginibre HW}.
Let us write
\begin{equation} \label{omega m}
    w_m(\zeta) := (h_m)^{-\frac{1}{2}}\zeta^m e^{-N|\zeta|^2/2}
\end{equation}
for the weighted orthonormal polynomial, where $h_m$ is the orthogonal norm given by \eqref{hk Qradial}.
Since $Q$ is radially symmetric, it follows from \eqref{skew op_rad}, \eqref{kappa rescaling} and \eqref{bfkappaN skewOP} that the rescaled and weighted pre-kernel $\kappa_N$ can be written as
\begin{equation} \label{kappa GN}
    \kappa_N(z,w) e^{-N(|\zeta|^2 + |\eta|^2)/2} = G_N(z,w) - G_N(w,z),
\end{equation}
where
\begin{equation} \label{GNsumsh}
 G_N(z,w):= \frac{\gamma_N^3}{2} \sum_{k=0}^{N-1} w_{2k+1}(\zeta)\Big[ \Big(\frac{h_{2k}}{h_{2k+1}}\Big)^{\frac{1}{2}}w_{2k}(\eta) + \sum_{l=0}^{k-1}\Big(\frac{h_{2l}}{h_{2k+1}}\Big)^{\frac{1}{2}}\prod_{j=0}^{k-l-1}\frac{h_{2l+2j+2}}{h_{2l+2j+1}}w_{2l}(\eta) \Big].
\end{equation}
Here and in the sequel, we denote $\zeta = p+\gamma_N z$ and $\eta = p+\gamma_N w$, where $\gamma_N$ is the microscopic scale given by \eqref{microscale Laplacian} for the soft/hard edge case and by \eqref{microscale HW} for the hard edge case.
Throughout this section, we write
\begin{equation} \label{delta n tau m}
\delta_N = N^{-\frac{1}{2}}\log N, \qquad \tau_m = \frac{m}{N}.
\end{equation}
Recall also that $D_\rho :=  \{z\in \C: |z|\leq \sqrt{2}\rho\}$.

The general strategy of deriving the large-$N$ behaviour of $G_N$ is as follows:
\begin{itemize}
    \item for $m \in \{1,2,\dots, 2N-1\}$, we compute the asymptotic behaviours of $w_m$ by means of Laplace's method;
    \smallskip
    \item we show that in the case of soft/hard edge ensemble, the dominant terms of $G_N$ near $\pa  D_1$ consist of $w_m$ with $2(1-\delta_N) \leq \tau_m \leq 2$, whereas in the case of hard edge ensemble, those near $\pa  D_\rho$ consist of $w_m$ with $ 2(\rho^2 + \delta_N) \leq \tau_m \leq 2$;
    \smallskip
    \item by discarding the lower degree terms, we compile the contributions of dominant terms using the Riemann sum approximation.
\end{itemize}

We now present some asymptotic behaviours of $h_m$.
For this purpose, we define
\begin{equation}\label{Vtau}
V_{\tau}(r) := r^2 - 2\tau \log r,\qquad r>0.
\end{equation}
Observe that the function $V_\tau$ has a unique critical point at $r_\tau := \sqrt{\tau}$.
It is convenient to write
\begin{equation}\label{Phi erfc}
    \Phi(x) := \int_{x}^{\infty} e^{-t^2}\,dt = \frac{\sqrt{\pi}}{2}\erfc(x).
\end{equation}

We first obtain the following. (See \cite[Lemmas 3.1, 3.3]{seo2020edge}) for a related statement.)

\begin{lem}\label{lem:henorm}  \label{lem:hardwallnorm}
For a given $\rho \in (0,1]$, let  $Q(\zeta)=|\zeta|^2 + \infty \cdot 1_{\C \setminus D_\rho}$. Then we have the following.
\begin{itemize}
    \item[(i)] For each $m$ with $|\tau_m - 2\rho^2|<2\delta_N$, we have
\begin{equation}
    h_m = e^{-m}\Big(\frac{m}{N}\Big)^{m}\frac{2\rho}{\sqrt{N}}\, \Phi(\xi_m)(1+o(1)), \qquad  \xi_m := \tfrac{\sqrt{N}}{2\rho}(\tau_m-2\rho^2),
\end{equation}
where $o(1)\to 0$ uniformly for $m$ as $N\to \infty$. \smallskip
    \item[(ii)] For each $m$ with $ 2(\rho^2 + \delta_N) \leq \tau_m \leq 2$, we have
\begin{equation}
    h_m = \frac{(2\rho^2)^{m+1}}{N(\frac{m}{N}-2\rho^2)}e^{-2\rho^2N}(1+o(1)),
\end{equation}
where $o(1)\to 0$ uniformly for $m$ as $N\to\infty$.
\end{itemize}
\end{lem}

\begin{proof}
For $m$ with $|\tau_m - 2\rho^2|<2\delta_N$, the orthogonal norm $h_m$ can be written as
\begin{equation}
\label{h:divide}
    h_m:=\int_{D_\rho}|\zeta|^{2m} e^{-N|\zeta|^2} \,dA(\zeta)
    = 2\int_{r_{\tau_m}-\delta_N'}^{\sqrt{2}\rho} e^{-NV_{\tau_m}(r)} r\,dr + 2\int_{0}^{r_{\tau_m}-\delta_N'} e^{-NV_{\tau_m}(r)}r\,dr,
\end{equation}
where $\delta_N' := N^{-\frac{1}{2}}(\log N)^2$.
By the Taylor expansion
\begin{equation}
    V_\tau(r) = V_\tau(r_\tau) + 2(r-r_\tau)^2 + O(r-r_\tau)^3,\qquad r\to r_\tau
\end{equation}
at the critical point $r = r_\tau$, we obtain that as $N \to \infty$,
\begin{equation}
\label{h:first}
    2\int_{r_{\tau_m}-\delta_N'}^{\sqrt{2}\rho} r e^{-NV_{\tau_m}(r)} \, dr = 2e^{-NV_{\tau_m}(r_{\tau_m})}(\sqrt{2}\rho + O(\delta_N'))\int_{r_{\tau_m}-\delta_N'}^{\sqrt{2}\rho} e^{-2N(r-r_{\tau_m})^2 + O(N\delta_N'^3)} \, dr.
\end{equation}
Here, $O$-constant can be taken independent of $m$.
Note that by \eqref{delta n tau m},
\begin{equation*}
     r_{\tau_m}- \sqrt{2}\rho =  \sqrt{\tau_m} -\sqrt{2}\rho = \frac{1}{2\sqrt{2}\rho}(\tau_m-2\rho^2)+O(\delta_N^2),\qquad N\to\infty.
\end{equation*}
Therefore, by the change of variable $t = \sqrt{2N}(r-r_{\tau_m})$, we have
\begin{equation}\label{h:first2}
    \int_{r_{\tau_m}-\delta_N'}^{\sqrt{2}\rho} e^{-2N(r-r_{\tau_m})^2} \, dr = \frac{1}{\sqrt{2N}}\int_{\xi_m}^{\infty} e^{-t^2} \, dt \, (1+o(1))
\end{equation}
as $N\to \infty$.
Since $V_{\tau_m}(r_{\tau_m}) = \frac{m}{N} - \frac{m}{N}\log\frac{m}{N}$, it follows from \eqref{h:first} and \eqref{h:first2} that
\begin{equation}\label{h:first3}
    2\int_{r_{\tau_m}-\delta_N'}^{\sqrt{2}\rho} e^{-NV_{\tau_m}(r)} r\,dr = e^{-m}\Big(\frac{m}{N}\Big)^m\, \frac{2\rho}{\sqrt{N}}\,\Phi(\xi_m)\,(1+o(1)),
\end{equation}
where $o(1)\to 0$ uniformly for $m$ as $N\to \infty$.

On the other hand, the second integral on the right-hand side of \eqref{h:divide} is negligible.
To see this, observe first that $V'_{\tau_m}(r_{\tau_m})=0$ and
$V'_{\tau_m}(r) < 0$ for $r<r_{\tau_m}$. Thus for $r<r_{\tau_m}-\delta_N'$
\begin{equation*}
    V_{\tau_m}(r) \geq V_{\tau_m}(r_{\tau_m}-\delta_N') = V_{\tau_m}(r_{\tau_m}) + C\delta_N'^2
\end{equation*}
for some positive constant $C$.
This gives that as $N \to \infty,$
\begin{equation}
    2\int_{0}^{r_{\tau_m}-\delta_N'} e^{-NV_{\tau_m}(r)}r\,dr \leq C'e^{-NV_{\tau_m}(r_{\tau_m})}e^{-CN\delta_N'^2} = e^{-m}\Big(\frac{m}{N}\Big)^{m} O(e^{-C(\log N)^4}).
\end{equation}
We have shown the first assertion of the lemma.

Next, we show the second assertion.
The proof is similar to that of Lemma~\ref{lem:henorm} (i).
A key observation is that for $m$ with $2N(\rho^2 + \delta_N)\leq m\leq 2N$
\begin{align*}
   h_m &= \int_{\sqrt{2}\rho -\delta_N}^{\sqrt{2}\rho} 2r^{m+1} e^{-Nr^2}\, dr \cdot (1+o(1)) \\
    &= 2\sqrt{2}\rho\, e^{-NV_{\tau_m}(\sqrt{2}\rho)} \int_{0}^{\delta_N}  e^{-\frac{\sqrt{2}N}{\rho}(\frac{m}{N}-2\rho^2) t}\, dt \cdot(1+o(1)) = \frac{(2\rho^2)^{m+1}}{N(\frac{m}{N}-2\rho^2)} e^{-2\rho^2N}\cdot (1+o(1)),
\end{align*}
which completes the proof.
\end{proof}

\subsection{Non-Hermitian soft/hard edge scaling limit} \label{Subsec_SH edge}

In this subsection, we prove Theorem~\ref{Thm_NH hE}.
We first show the following proposition (here $\H_{-} = \{ z \in \C : \Re z < 0 \}$ denotes the left half plane).

\begin{prop}\label{thm:HES}
There exists a sequence of cocycles $(c_N)_{N \ge 1}$ such that
\begin{equation*}
    \lim_{N\to\infty} c_N(z,w)G_N(z,w) = 2e^{z^2 - |z|^2 + w^2 - |w|^2} \int_{-\infty}^{0} \frac{e^{-2(z - s)^2}}{\sqrt{\Phi(2s)}} \int_{-\infty}^{s} \frac{e^{-2(w - t)^2}}{\sqrt{\Phi(2t)}} dt \,ds,
\end{equation*}
where the convergence is uniform for $z,w$ in compact subsets of $\H_{-}$. Here, the function $\Phi$ is given by \eqref{Phi erfc}.
\end{prop}

Using Proposition~\ref{thm:HES}, we first complete the proof of  Theorem~\ref{Thm_NH hE}.
\begin{proof}[Proof of  Theorem~\ref{Thm_NH hE}. ]
Recall that the potential $Q$ is given by \eqref{Q Ginibre HE}.
Then by combining Proposition~\ref{thm:HES} with \eqref{kappa GN} and  \eqref{RNk rescaling}, the theorem follows. Here, we again use the fact that a sequence of cocycles cancels out when forming a Pfaffian.
\end{proof}

We first discard the negligible terms in the sum \eqref{GNsumsh}. Recall that $\delta_N$ and $\tau_m$ are given by \eqref{delta n tau m}.

\begin{lem}\label{lem:upperwm}
For $m$ with $|\tau_m -2|\leq 2\delta_N$ we have
\begin{equation}\label{w:asym}
    w_m(\zeta) = \Big(\frac{N}{4}\Big)^{\frac{1}{4}}(\Phi(\xi_m))^{-\frac{1}{2}}e^{2i\sqrt{N}\im z} e^{z^2 - |z|^2} e^{-2(z - \xi_m/2)^2}(1+o(1)),
\end{equation}
where $\xi_m = \sqrt{N}(\frac{m}{2N}-1)$. Here, $o(1)\to 0$ uniformly for all $z$ in any compact set as $N\to \infty$.
\end{lem}
\begin{proof}
By Lemma \ref{lem:henorm} (i), we have
\begin{equation*}
    w_m(\zeta) = \Big(\frac{N}{4}\Big)^{\frac{1}{4}}(\Phi(\xi_m))^{-\frac{1}{2}}e^{NV_{\tau_m}(r_{\tau_m})/2}e^{-N(|\zeta|^2 - \frac{2m}{N}\log \zeta)/2}(1+o(1)),
\end{equation*}
where $V_{\tau}$ is given in \eqref{Vtau} and $r_{\tau_m} = \sqrt{\frac{m}{N}}$. Since
\begin{equation*}
|\zeta|^2 = |p+\gamma_Nz|^2 = (p+\gamma_N z)^2 + \gamma_N^2(|z|^2 - z^2) - 2ip\gamma_N \im z = \zeta^2 + \frac{2}{N}(|z|^2-z^2)- \frac{4i}{\sqrt{N}}\im z ,
\end{equation*}
we obtain that
\begin{align*}
    e^{-N(|\zeta|^2 - \frac{2m}{N}\log \zeta)/2} = e^{z^2 - |z|^2} e^{2i\sqrt{N}\im z} e^{-N(\zeta^2 - \frac{2m}{N}\log \zeta)/2}.
\end{align*}
Now the lemma follows from straightforward computations using the Taylor series expansion at $r=r_{\tau_m}$
\begin{equation*}
    \zeta^2 - \frac{2m}{N}\log \zeta = V_{\tau_m}(r_{\tau_m}) + 2(\zeta - r_{\tau_m})^2 + O(|\zeta-r_{\tau_m}|^3)
\end{equation*}
and the following asymptotic expansion
\begin{equation*}
    \zeta - r_{\tau_m} = p+\gamma_Nz - r_{\tau_m} = \sqrt{\frac{2}{N}}z + \sqrt{2}\Big(1-\sqrt{\frac{m}{2N}}\Big)=
    \sqrt{\frac{2}{N}}z + \frac{1}{\sqrt{2}}\Big(1-{\frac{m}{2N}}\Big)
    +O(\delta_N^2).
\end{equation*}
\end{proof}

\begin{lem}\label{lem:lowerwm}
Let $K$ be a compact subset of $\H_{-}$. Then there exist positive constants $c$ and $C$ such that for all $m$ with $m \leq 2N(1-\delta_N)$ and for all $z \in K$,
\begin{equation}
|w_m(\sqrt{2}+\gamma_N z)|^2 \leq Ce^{-c(\log N)^2}.
\end{equation}
\end{lem}

\begin{proof}
For $m \leq 2N(1-\delta_N)$, the critical point $r_{\tau_m}$ of $V_{\tau_m}$ satisfies
\begin{equation}
    r_{\tau_m} = \sqrt{\tau_{m}} \leq \sqrt{2(1-\delta_N)} = \sqrt{2}\Big(1-\frac{1}{2}\delta_N\Big)+O(\delta_N^2), \qquad N \to \infty.
\end{equation}
Choose a constant $c_1>0$ such that $r_{\tau_m}\leq \sqrt{2} - 2c_1\delta_N$.
Since $V'_{\tau_m}(r)\geq 0$ for $r\geq r_{\tau_m}$, we obtain
\begin{equation}
    h_m  \geq \int_{\sqrt{2}-2c_1\delta_N}^{\sqrt{2}-c_1\delta_N} 2r  e^{-NV_{\tau_m}(r)} dr \geq C_1\delta_N e^{-NV_{\tau_m}(\sqrt{2}-c_1\delta_N)}
\end{equation}
for some constant $C_1$.
By \eqref{omega m}, this gives that for $z \in K$,
\begin{equation*}
    |w_m(\sqrt{2}+\gamma_N z)|^2 \leq C_2 \delta_N^{-1} e^{-N(V_{\tau_m}(|\sqrt{2}+\gamma_N z|)-V_{\tau_m}(\sqrt{2}-c_1\delta_N))}\leq C e^{-c(\log N)^2}
\end{equation*}
for some constants $c$ and $C$.
\end{proof}

Lemma \ref{lem:lowerwm}  asserts that for $m\leq 2N(1-\delta_N)$, $w_m$ is negligible in the sum \eqref{GNsumsh}. Indeed, we obtain the following lemma.

\begin{lem} \label{Lem_G:asym}
Let $K$ be a compact subset of $\mathbb{H}_-$. Then for $z,w\in K$, we have
\begin{equation}\label{G:asym}
    G_N(z,w) \sim \frac{\gamma_N^3}{2}\sum_{ k=N(1-\delta_N)}^{N-1} w_{2k+1}(\zeta) \Big[ \Big(\frac{h_{2k}}{h_{2k+1}}\Big)^{\frac{1}{2}}w_{2k}(\eta) + \sum_{l=N(1-\delta_N)}^{k-1}\Big(\frac{h_{2l}}{h_{2k+1}}\Big)^{\frac{1}{2}}\prod_{j=0}^{k-l-1}\frac{h_{2l+2j+2}}{h_{2l+2j+1}}w_{2l}(\eta) \Big].
\end{equation}
\end{lem}

\begin{proof}
We first mention that the number of summands $N \delta_N= N^{1/2} \log N \to \infty$ as $N \to \infty$.
For each $m$, the norm $h_m$ can be expressed in terms of the lower incomplete function $\gamma$ as
\begin{align*}
    h_{m} = \int_{0}^{\sqrt{2}} 2r^{2m+1} e^{-Nr^2}\, dr = N^{-(m+1)}\gamma(m+1,2N).
\end{align*}
It is well known that if $X_N$ is a $\mathrm{Poisson}(2N)$ random variable, then
\begin{equation*}
    \frac{\Gamma(m+1) - \gamma(m+1,2N)}{\Gamma(m+1)} = \P[X_N<m+1].
\end{equation*}
For all $m<2N(1-\delta_N)$, the normal approximation of Poisson random variables gives
\begin{equation*}
    \P[X_N\leq m] \leq \P\Big[\frac{X_N-2N}{\sqrt{2N}}\leq - \sqrt{2}\log N\Big]= \P[Z \leq -\sqrt{2}\log N](1+o(1)),
\end{equation*}
where $Z$ is a standard normal random variable.
Note that $\P[Z \leq -2 \log N] = O(e^{-c(\log N)^2})$ for some constant $c$. This implies that there exists a constant $c'$ such that
\begin{equation}
 \frac{h_{m}}{h_{m+1}}=\frac{N}{m+1}(1+O(e^{-c'(\log N)^2})), \qquad   \frac{h_{m}}{h_{m+1}}\frac{h_{m+2}}{h_{m+1}} = \frac{m+2}{m+1}(1+O(e^{-c'(\log N)^2})),
\end{equation}
where $O(e^{-c'(\log N)^2}) \to 0 $ uniformly for $m$ as $N\to \infty$. This implies that
$   {h_{m}}/{h_{m+1}} = O(N) $
and for $k < N(1-\delta_N)$
\begin{equation}
     \sqrt{\frac{{h_{2l}}}{{h_{2k+1}}}}\prod_{j=0}^{k-l-1}\frac{h_{2l+2j+2}}{h_{2l+2j+1}} = \sqrt{\frac{h_{2k}}{h_{2k+1}}}\prod_{j=0}^{k-l-1}\sqrt{\frac{h_{2l+2j+2}}{h_{2l+2j+1}}}\sqrt{\frac{h_{2l+2j}}{h_{2l+2j+1}}} = O(N).
\end{equation}

Combining all of the above with Lemma \ref{lem:lowerwm}, we obtain
\begin{equation}
    \sum_{k=0}^{N(1-\delta_N)} w_{2k+1}(\zeta)\Big[ \Big(\frac{h_{2k}}{h_{2k+1}}\Big)^{\frac{1}{2}}w_{2k}(\eta) + \sum_{l=0}^{k-1}\Big(\frac{h_{2l}}{h_{2k+1}}\Big)^{\frac{1}{2}}\prod_{j=0}^{k-l-1}\frac{h_{2l+2j+2}}{h_{2l+2j+1}}w_{2l}(\eta) \Big] = o(1)
\end{equation}
uniformly for $z,w \in K$. Similarly, we obtain
\begin{equation}
    \sum_{k=N(1-\delta_N)}^{N-1} w_{2k+1}(\zeta)  \sum_{l=0}^{N(1-\delta_N)}\Big(\frac{h_{2l}}{h_{2k+1}}\Big)^{\frac{1}{2}}\prod_{j=0}^{k-l-1}\frac{h_{2l+2j+2}}{h_{2l+2j+1}}w_{2l}(\eta) = o(1),
\end{equation}
which completes the proof.
\end{proof}

We now show the following.
\begin{lem} \label{Lem_hk prod}
For $k,l$ with $N (1-\delta_N) \leq k,l \leq N-1$, we have
\begin{equation}
 \frac{1}{\sqrt{h_{2k+1}}} = \frac{1}{\sqrt{2h_{2k}}} (1+o(1)), \qquad
    \frac{1}{\sqrt{h_{2k+1}}}\prod_{j=0}^{k-l-1}\frac{h_{2l+2j+2}}{h_{2l+2j+1}} = \frac{1}{\sqrt{2h_{2l}}}(1+o(1)).
\end{equation}
\end{lem}
\begin{proof}
Similar to the proof of Lemma \ref{Lem_G:asym}, the ratio of the norms can be estimated by
\begin{equation}\label{eq:shratio}
    \frac{h_{m+1}}{h_m} = \frac{m+1}{N}\cdot\frac{1-\P [X_N\leq m+1]}{1-\P[X_N \leq m]}=\frac{m+1}{N}\Big(1-\frac{\Phi(\xi_m)-\Phi(\xi_{m+1})}{\Phi(\xi_{m})}(1+o(1))\Big)
\end{equation}
as $m,N\to \infty$ with $|m-2N|
\leq 2N\delta_N$. Here, $X_N \sim \mathrm{Poisson}(2N)$ and $\Phi$, $\xi_m$ are defined in \eqref{Phi erfc} and Lemma \ref{lem:henorm} (i).
Since there exists $c_{m} \in (\xi_m, \xi_{m+1})$ such that
\begin{equation}\label{eq:Phidiff}
    \Phi(\xi_m) - \Phi(\xi_{m+1}) = \frac{1}{2\sqrt{N}}e^{-c_{m}^2}
\end{equation}
by the mean value theorem, we have for $k$ with $N(1-\delta_N)\leq k \leq N-1$,
\begin{equation}
    \frac{h_{2k+1}}{h_{2k}} = 2 \cdot (1+o(1)),
\end{equation}
where $o(1)\to 0$ uniformly for $m$ as $N\to \infty$.
This gives the first assertion.

For the second asymptotic equation, we obtain from \eqref{eq:shratio} and \eqref{eq:Phidiff} that
\begin{equation*}
    \frac{h_{m}}{h_{m+1}}\frac{h_{m+2}}{h_{m+1}}   =  1+ O(N^{-1})
\end{equation*}
for $m$ with $2N(1-\delta_N) \leq m \leq 2N$.
This gives
\begin{equation*}
    \frac{1}{\sqrt{h_{2k+1}}}\prod_{j=0}^{k-l-1}\frac{h_{2l+2j+2}}{h_{2l+2j+1}} = \sqrt{\frac{h_{2k}}{h_{2k+1}h_{2l}}}\prod_{j=0}^{k-l-1}\sqrt{\frac{h_{2l+2j+2}}{h_{2l+2j+1}}}\sqrt{\frac{h_{2l+2j}}{h_{2l+2j+1}}} = \frac{1}{\sqrt{2h_{2l}}} (1+o(1)),
\end{equation*}
which yields the desired equation.
\end{proof}

We are now ready to show Proposition~\ref{thm:HES}.

\begin{proof}[Proof of Proposition~\ref{thm:HES}]
By Lemmas~\ref{lem:upperwm}, \ref{Lem_G:asym} and ~\ref{Lem_hk prod}, we have
\begin{equation}
    G_N(z,w) \sim \frac{e^{2i\sqrt{N}\im (z+ w)}}{2N}  e^{z^2 - |z|^2 + w^2 - |w|^2} \sum_{k=N(1-\delta_N)}^{N-1} \frac{e^{-2(z - \xi_{2k+1}/2)^2}}{\sqrt{\Phi(\xi_{2k+1})}}\sum_{l=N(1-\delta_N)}^{k}\frac{e^{-2(w-\xi_{2l}/2)^2}}{\sqrt{\Phi(\xi_{2l})}}.
\end{equation}
Then the proposition follows from the Riemann sum approximation.
\end{proof}

\subsection{Non-Hermitian hard edge scaling limit}
\label{Subsec_H edge}

This subsection is devoted to the proof of Theorem~\ref{Thm_NH hW}.
As in the previous subsection, by \eqref{kappa GN}, it is enough to show the following proposition.

\begin{prop}\label{thm:HWS}
There exists a sequence of cocycles $(c_N)_{N \ge 1}$ such that
\begin{equation*}
    \lim_{N\to\infty} c_N(z,w)\,G_N(z,w) =  \int_{0}^{1} s^{\frac{1}{2}} e^{2sz} \int_{0}^{s} t^{\frac{1}{2}} e^{2tw} \,dt \, ds,
\end{equation*}
where the convergence is uniform for $z,w$ in any compact subset of $\H_{-}$.
\end{prop}

We first show the following lemma.

\begin{lem}\label{lem:wmhard}
For each $m$ with $(2\rho^2 + 2\delta_N)N < m <2N$ and for all $z$ in a compact set, we have
\begin{equation*}
    w_m(\zeta) =
    \frac{g_N(z)}{\rho}\sqrt{N\Big(\frac{m}{2N}-\rho^2\Big)}
    e^{\frac{2}{1-\rho^2}(\frac{m}{2N}-\rho^2)z}
    (1+o(1)),
\end{equation*}
where $g_N(z)$ is a unimodular function.
\end{lem}
\begin{proof}
By Lemma \ref{lem:hardwallnorm} (ii), we obtain
\begin{equation*}
    w_m(\zeta) = \frac{1}{\rho}\sqrt{N\Big(\frac{m}{2N}-\rho^2\Big)}
    e^{-N(|\zeta|^2-2\tau_m{N}\log \zeta -V_{\tau_m}(\sqrt{2}\rho))/2}(1+o(1)),
\end{equation*}
where $\tau_m = m/N$ and $V_{\tau_m}(r) = r^2 -2\tau_{m}\log r$ defined in \eqref{Vtau}.
Using the expansion
\begin{align*}
    &\quad |\sqrt{2}\rho + \gamma_N z|^2-2\tau_m\log (\sqrt{2}\rho + \gamma_N z) - V_{\tau_m}(\sqrt{2}\rho) \\
    &= 2\sqrt{2}\rho \gamma_N \re z -\frac{2m}{\sqrt{2}\rho N}\gamma_N z + O(N^{-2})
    = \frac{4}{N(1-\rho^2)}\Big(\rho^2 - \frac{m}{2N}\Big)z-i\frac{4\rho^2}{N(1-\rho^2)}\im z + O(N^{-2}),
\end{align*}
we have the following asymptotic expansion:
\begin{equation*}
    w_m(\sqrt{2}\rho + \gamma_N z) =
    \frac{1}{\rho}\sqrt{N\Big(\frac{m}{2N}-\rho^2\Big)}\,
     e^{i\frac{2\rho^2}{1-\rho^2}\im z} e^{\frac{2}{1-\rho^2}(\frac{m}{2N}-\rho^2)z}
    (1+o(1)).
\end{equation*}
\end{proof}

\begin{lem}\label{lem:lowerwmhe}
Let $K$ be a compact subset of $\H_{-}$.
Then there exists a positive constant $C$ such that for all $m$ with $|\tau_m - 2\rho^2|<2\delta_N$ and for all $z\in K$,
\begin{equation*}
    |w_m(p+\gamma_N z)|^2 \leq C\sqrt{N}\log N.
\end{equation*}
\end{lem}
\begin{proof}
By Lemma~\ref{lem:henorm} (i), we have
\begin{equation}
    |w_m(p+\gamma_N z)|^2 = (h_m)^{-1} e^{-NV_{\tau_m}(|p+\gamma_N z|)} = \frac{\sqrt{N}}{2\rho\, \Phi(\xi_m)} e^{-N(V_{\tau_m}(|p+\gamma_N z|)-V_{\tau_m}(r_{\tau_m}))}(1+o(1)),
\end{equation}
where $\xi_m = \frac{\sqrt{N}}{2\rho}(\tau_m - 2\rho^2)$ and $r_{\tau_m}$ is the critical point of $V_{\tau_m}$. The Taylor expansion of $V_{\tau_m}$ at $r_{\tau_m}$ gives
\begin{equation}
     |w_m(p+\gamma_N z)|^2 = \frac{\sqrt{N}}{2\rho\,\Phi(\xi_m)}e^{-2N(p+\gamma_N\re z-r_{\tau_m})^2}(1+o(1)) = \frac{\sqrt{N}}{2\rho\,\Phi(\xi_m)}e^{-\xi_m^2}(1+o(1))
\end{equation}
since $p - r_{\tau_m} = \sqrt{2}\rho - \sqrt{\frac{m}{N}}=\frac{\xi_m}{\sqrt{2N}}+O(N^{-1}|\xi_m|^2)$ as $N\to \infty$. Here $o(1)
\to 0$ uniformly for $z\in K$. Note that $|\xi_m| = O(\log N)$.
By \eqref{Phi erfc}, it follows from the well-known asymptotics of the complementary error function (see e.g.~\cite[Section 7.12]{olver2010nist}) that
\begin{equation*}
e^{\xi_m^2}\,\Phi(\xi_m) \geq C'(\log N)^{-1}.
\end{equation*}
This completes the proof.
\end{proof}

As in Section \ref{Subsec_SH edge}, the following lemma asserts that the weighted polynomials $w_m$ of lower degree $m<2N(\rho^2 - \delta_N)$ tend to exponentially decay in a neighbourhood of the hard edge. For the proof of the lemma, we refer to Lemma 3.5 in \cite{seo2020edge}.
\begin{lem}\label{lem:lowerlower}
Let $K$ be a compact subset of $\H_{-}$. Then for $m$ with $0\leq m<N(2\rho^2 - 2\delta_N)$ and for $z\in K$, we have
\begin{equation}
    |w_m(p+\gamma_N z)|^2\leq C e^{-c(\log N)^2}),
\end{equation}
where $c$ and $C$ are positive constants.
\end{lem}

As a counterpart of Lemma~\ref{Lem_G:asym}, we obtain the following lemma which allows us to discard the lower degree polynomials in the sum \eqref{GNsumsh}.

\begin{lem} \label{Lem_G:asym 2}
Let $K$ be a compact subset of $\mathbb{H}_-$. Then for $z,w\in K$, we have
\begin{equation}\label{eq:GNhard}
    G_N(z,w) \sim \frac{\gamma_N^3}{2}\!\! \sum_{k = (\rho^2+\delta_N)N}^{N-1}\! w_{2k+1}(\zeta) \Big( \Big(\frac{h_{2k}}{h_{2k+1}}\Big)^{\frac{1}{2}}w_{2k}(\eta) + \!\!\sum_{ l=(\rho^2+\delta_N)N}^{k-1}\!
    \Big( \frac{h_{2l}}{h_{2k+1}} \Big)^{\frac{1}{2}} \prod_{j=0}^{k-l-1} \frac{h_{2l+2j+2}}{h_{2l+2j+1}}w_{2l}(\eta)\Big).
\end{equation}
\end{lem}
\begin{proof}
Similarly as in the proof of Lemma~\ref{Lem_G:asym}, we obtain for all $m$ with $m<2N(\rho^2 - \delta_N)$
\begin{equation}
    \frac{h_{m}}{h_{m+1}} = \frac{N\,\gamma(m+1,2\rho^2N)}{\gamma(m+2,2\rho^2N)} = \frac{N}{m+1}(1+O(e^{-c(\log N)^2})), \quad   \frac{h_{m}}{h_{m+1}}\frac{h_{m+2}}{h_{m+1}} = \frac{m+2}{m+1}(1+O(e^{-c(\log N)^2}))
\end{equation}
for some constant $c$.
Together with Lemma~\ref{lem:lowerlower}, this implies for all $z,w\in K$
\begin{equation}
    \sum_{k=0}^{N(\rho^2-\delta_N)} w_{2k+1}(\zeta)\Big[ \Big(\frac{h_{2k}}{h_{2k+1}}\Big)^{\frac{1}{2}}w_{2k}(\eta) + \sum_{l=0}^{k-1}\Big(\frac{h_{2l}}{h_{2k+1}}\Big)^{\frac{1}{2}}\prod_{j=0}^{k-l-1}\frac{h_{2l+2j+2}}{h_{2l+2j+1}}w_{2l}(\eta) \Big] = o(1)
\end{equation}
since all the weighted polynomials $w_{2k+1}$ and $w_{2l}$ in the sum are exponentially small as $N\to \infty$.

Now for $m$ with $|\tau_m - 2\rho^2 |<2\delta_N$, it follows from Lemma~\ref{lem:henorm} (i) (see also Lemma~\ref{Lem_hk prod}) that
\begin{equation}
    \frac{h_{m}}{h_{m+1}} = \frac{1}{2\rho^2}(1+O(N^{-\frac{1}{2}})),\qquad \frac{h_{m}}{h_{m+1}}\frac{h_{m+2}}{h_{m+1}} = 1+O(N^{-1}).
\end{equation}
This gives the estimate for $k,l$ with $|k-\rho^2N|<N\delta_N$ and $|l-\rho^2N|<N\delta_N$
\begin{equation*}
    \sqrt{\frac{h_{2l}}{{h_{2k+1}}}}\prod_{j=0}^{k-l-1}\frac{h_{2l+2j+2}}{h_{2l+2j+1}} = \sqrt{\frac{h_{2k}}{h_{2k+1}}}\prod_{j=0}^{k-l-1}\sqrt{\frac{h_{2l+2j+2}}{h_{2l+2j+1}}}\sqrt{\frac{h_{2l+2j}}{h_{2l+2j+1}}} = \frac{1}{\sqrt{2\rho^2}} (1+o(1)).
\end{equation*}
Note that Lemma~\ref{lem:lowerwmhe} gives a bound $|w_m(
\eta)|^2 < C\sqrt{N}\log N$ for all $m$ with $|\tau_m - 2\rho^2|<2\delta_N$ and all $z\in K$.
Since the sum
\begin{equation}\label{sum:interm}
    \gamma_N^3 \sum_{k=N(\rho^2-\delta_N)}^{N(\rho^2+\delta_N)} w_{2k+1}(\eta)\Big[ \Big(\frac{h_{2k}}{h_{2k+1}}\Big)^{\frac{1}{2}}w_{2k}(\eta) + \sum_{l=0}^{k-1}\Big(\frac{h_{2l}}{h_{2k+1}}\Big)^{\frac{1}{2}}\prod_{j=0}^{k-l-1}\frac{h_{2l+2j+2}}{h_{2l+2j+1}}w_{2l}(\eta) \Big]
\end{equation}
contains $O(N^2\delta_N)$ terms and $\gamma_N=O(N^{-1})$,
the sum \eqref{sum:interm} has a bound of $O(N^{-1}\log^2 N)$.

Finally, we obtain that
\begin{equation}
    \gamma_N^3 \sum_{k=N(\rho^2+\delta_N)}^{N-1} w_{2k+1}(\eta) \sum_{l=0}^{N(\rho^2+\delta_N)}\Big(\frac{h_{2l}}{h_{2k+1}}\Big)^{\frac{1}{2}}\prod_{j=0}^{k-l-1}\frac{h_{2l+2j+2}}{h_{2l+2j+1}}w_{2l}(\eta) =
    O(N^{-\frac{3}{4}}\log^2 N)
\end{equation}
since $|w_{2k+1}|^2$ is at most $O(N)$ for $k \geq  N(\rho^2 + \delta_N)$, $|w_{2l}|^2 = O(\sqrt{N}\log N)$ for $|l - \rho^2N|\leq N\delta_N$, and $w_{2l}$ with $l \leq N(\rho^2 - \delta_N)$ is exponentially small.
Hence, the proof is complete.
\end{proof}

\begin{lem}\label{lem:hwall}
For $k$, $l$ with $l\leq k-1$ and
$(\rho^2 + \delta_N)N \leq k,l < N$, we have
\begin{equation}\label{hratio:wall}
    \sqrt{\frac{h_{2k+1}}{h_{2k}}} = \sqrt{2}\rho\,(1+o(1)),  \qquad
    \sqrt{\frac{h_{2l}}{h_{2k+1}}}\prod_{j=0}^{k-l-1}\frac{h_{2l+2j+2}}{h_{2l+2j+1}} = \frac{1}{\sqrt{2}\rho}\, (1+o(1)),
\end{equation}
where $o(1)\to 0$ as $N\to \infty$.
\end{lem}
\begin{proof}
By Lemma~\ref{lem:hardwallnorm} (ii), the ratio of the norm in \eqref{hratio:wall} is obtained directly from the following observation: for $(2\rho^2  + 2\delta_N)N \leq m <2N$
\begin{equation*}
    \frac{h_{m+1}}{h_m} = 2\rho^2 \frac{m-2\rho^2 N}{m+1 - 2\rho^2 N} (1+o(1)) = 2\rho^2 (1+o(1)).
\end{equation*}
For the latter asymptotic expansion in \eqref{hratio:wall}, we see that
\begin{align*}
    \sqrt{\frac{h_{2l}}{h_{2k+1}}}\prod_{j=0}^{k-l-1}\frac{h_{2l+2j+2}}{h_{2l+2j+1}}
    &= \sqrt{\frac{h_{2k}}{h_{2k+1}}}\prod_{j=0}^{k-l-1}\sqrt{\frac{h_{2l+2j+2}}{h_{2l+2j+1}}}\sqrt{\frac{h_{2l+2j}}{h_{2l+2j+1}}}\\
    &=\frac{1}{\sqrt{2}\rho}\prod_{m=l}^{k-1}\sqrt{\frac{2m+1-2\rho^2 N}{2m+2-2\rho^2 N} \cdot \frac{2m+1-2\rho^2 N}{2m -2\rho^2 N}}(1+o(1)).
\end{align*}
Here, the product can be approximated by $1$ since for $l,k$ with $(\rho^2  + \delta_N)N \leq l,k \leq N-1$
\begin{align*}
    \prod_{m=l}^{k-1}\frac{(2m+2-2\rho^2N)(2m-2\rho^2N)}{(2m+1-2\rho^2N)^2} = \prod_{m=l}^{k-1}\Big(1-\Big(\frac{1}{2m+1-2\rho^2 N}\Big)^2\Big) = 1 + o(1),
\end{align*}
which completes the proof.
\end{proof}

We now prove Proposition~~\ref{thm:HWS}.

\begin{proof}[Proof of Proposition~\ref{thm:HWS}]
Combining Lemmas~ \ref{lem:wmhard}, \ref{Lem_G:asym 2}, and \ref{lem:hwall}, we obtain
\begin{align*}
    &\quad G_N(z,w)
    \sim \frac{\gamma_N^3}{2\sqrt{2}\rho} \sum_{k=N(\rho^2 + \delta_N)}^{N-1} w_{2k+1}(\zeta)
    \sum_{l=N(\rho^2 + \delta_N)}^{k-1} w_{2l}(\eta)\\
    &\sim  \frac{ g_N(z) g_N(w)}{N^2(1-\rho^2)^3} \sum_{k=N(\rho^2 + \delta_N)}^{N-1}\Big(\frac{2k+1}{2N}-\rho^2\Big)^{\frac{1}{2}} e^{\frac{2}{1-\rho^2}(\frac{2k+1}{2N}-\rho^2)z} \!\!\!
    \sum_{l=N(\rho^2 + \delta_N) }^{ k-1} \!\! \Big(\frac{l}{N}-\rho^2\Big)^{\frac{1}{2}} e^{\frac{2}{1-\rho^2}(\frac{l}{N}-\rho^2)w}.
\end{align*}
Now the proposition follows from the Riemann sum approximation.
\end{proof}

\subsection*{Acknowledgement.} We wish to express our gratitude to Gernot Akemann, Nam-Gyu Kang, and Iván Parra for helpful discussions.

%%%%%%%%%%bibliography%%%%%%%%%%%%%%%%%%%%%%%%%%%%%%%%%%%
\bibliographystyle{abbrv}
\bibliography{RMTbib}
%%%%%%%%%%%%%%%%%%%%%%%%%%%%%%%%%%%%%%%%%%%%%%%%%%%%%%%%%%%%%%
\end{document}